\documentclass[11pt]{article}
\usepackage{a4wide}
\usepackage{amsmath, amsthm, xcolor}
\usepackage{amssymb}
\usepackage{graphicx,soul}
\usepackage[normalem]{ulem}
\usepackage{relsize}
\usepackage{enumerate}
\newcommand{\hidden}[1]{}
\usepackage[colorlinks, linkcolor=blue, anchorcolor=blue, citecolor=blue]{hyperref}
\usepackage{color}
\usepackage{float} %设置图片浮动位置的宏包
\usepackage{subfigure} %插入多图时用子图显示的宏包
\usepackage[mathscr]{euscript}
\usepackage{tikz}
\numberwithin{equation}{section}%公式按章节编号

\newtheorem{theorem}{Theorem}[section]
\newtheorem{proposition}[theorem]{Proposition}

\newtheorem{corollary}[theorem]{Corollary}
\newtheorem{lemma}[theorem]{Lemma}

\newtheorem{definition}[theorem]{Definition}
\theoremstyle{remark}
\newtheorem{remark}[theorem]{{Remark}}

\newcommand{\Lop}{\mathcal L}
\newcommand{\ind}{\mathbf 1}
\DeclareMathOperator{\Var}{Var}

\begin{document}

\title{Survivor Set of Gauss Map with a Hole at $1$}

%\title{Appendix to the Shrinking Target Problem for Matrix Transformations of Tori: revisted}

 \author{Cheng Liu \\ {\small\sc (Wuhan) } \and Lingmin Liao \\ {\small\sc (Wuhan) }}

\bigskip

\date{}

 \maketitle

\begin{abstract}
	Let $T$ be the Gauss map. For each $\alpha\in [0,1]$, let $E(\alpha)$ be the survivor set consisting of all $x\in [0,1]$ whose orbit $\{T^k(x):k\ge 0\}$ never enters the interval $(\alpha,1]$. We prove that the set-valued function $\alpha\mapsto E(\alpha)$ is increasing and locally constant almost everywhere. Moreover, we explicitly characterize all intervals where the function $\alpha\mapsto E(\alpha)$ remains constant. We also describe the points at which the map $\alpha\mapsto E(\alpha)$ is not locally constant. Further, we obtain the Hausdorff dimension of $E(\alpha)$ for all $\alpha\in [0,1]$. We prove that the function $\alpha\mapsto \dim_H(E(\alpha))$ forms a devil's staircase on $(0,1]$, but is not continuous at $0$. This discontinuity implies that the result of Raith [Acta Math. Univ. Comenian. 63 (1994), 39-53] fails for piecewise monotone interval maps with infinitely many branches. Finally, we give an application quantifying the set of points sharing identical $\alpha$-continued fraction expansion and regular continued fraction expansion.
\end{abstract}
$\textbf{Keywords}$: open dynamical system, survivor set, $\alpha$-continued fraction, graph directed Markov system, devil's staircase, transfer operator

\bigskip

\section{Introduction}
\subsection{Open dynamical system}
Given a dynamical system on a topological space and an open subset, it is natural to study the associated \textit{survivor set}, that is, the subset of all points which never enter this subset under forward iteration. In this framework, the theory of open dynamical systems is concerned with escape rates, conditionally invariant measures, and other closely related concepts; see for example \cite{PU17}.
\par
Let $f:[0,1]\rightarrow [0,1]$ be a map. Let $H$ be an open subset of $[0,1]$. Define the surviving set
\[E_f(H)=\left\{x\in [0,1]: f^k(x)\notin H,\ \forall k\ge 0\right\}.\]
Note that $E_f(H)$ varies based on the size and position of the hole $H$. Many studies have investigated the Hausdorff dimension $\dim_H(E_f(H))$ of the survivor set with varying $H$. In the 1980s', Urbański \cite{U86,U87} proved that for $f:x\mapsto 2x\ ({\rm mod}\ 1)$ and $H=(0,t)$, the function $t\mapsto \dim_H(E_f((0,t)))$ forms a devil's staircase. This result was extended in 2020 by Kalle, Kong, Langeveld and Li \cite{KKLL20} for $f:x\mapsto \beta x\ ({\rm mod}\ 1)$ and $H=(0,t)$ for any $\beta\in (1,2]$. In 2023, Langeveld and Samuel \cite{LS23} further generalized these results to $f:x\mapsto \beta x+\alpha\ ({\rm mod}\ 1)$. Regarding holes not near $0$, the relevant results are significantly scarcer. In 2024, Sun, Li and Ding \cite{SLD24} studied a function $f$ with a hole $H=(a, b)$ at the discontinuity point $c$, where $a\le c \le b$ and $a\neq b$. They proved that if $f$ is an expansive Lorenz map that has an ergodic absolutely continuous invariant measure, then for fixed $a\in [0,c]$, the topological entropy function $\lambda_f(a):b\mapsto h_{top}\left(f\vert_{E_f(a,b)}\right)$ forms a devil's staircase. In 2025, Sun and Li \cite{SL25} investigated the topological entropy of two bifurcation sets related to survivor set and extended some earlier results of Baker and Kong \cite{BK20}, Allaart and Kong \cite{AK23}.  
\par
All the above results are under the condition that $f$ is a  piecewise monotone interval map with finitely many branches. By a result of Raith \cite{R94}, the function $H\mapsto \dim_H(E_f(H))$ is continuous on $[0,1]$. Therefore, to prove that $t\mapsto \dim_H(E_f((0,t)))$ forms a devil's staircase, it suffices to prove that this function is locally constant almost everywhere and monotonic. In 2022, Carminati and Tiozzo \cite{CT22} considered the Gauss map $T$ with hole $H=(0,t)$. At first glance, this appears to be a problem for a piecewise monotonic map with infinitely many branches; however, for the purpose of computing $\dim_H(E_T(0,t))$, the hole at 0 effectively reduces the problem to finitely many branches. That is also the reason that Raith's result is applicable for Carminati and Tiozzo (see \cite[p. 26]{CT22}).
\par
In this paper, we consider the Gauss map $T$ with $H=(\alpha,1)$, where $\alpha\in [0,1]$. Due to the persistent influence of the infinite branches around $0$, the system constitutes a truly infinite piecewise monotone system. Define
\begin{align}\label{edy}
  E(\alpha)=\left\{x\in [0,1]: T^k(x)\le \alpha ,\ \forall k\ge 0\right\}.
\end{align}  
We first prove that the set valued function $\alpha\mapsto E(\alpha)$ is increasing and locally constant almost everywhere. Moreover, we characterize all intervals on which the map is constant. Denote
\[[a_1,a_2,\cdots,a_n]=\cfrac{1}{a_1+\cfrac{1}{\ddots +\cfrac{1}{a_n}}}\]
and
\[[a_1,\cdots,a_n,\overline{b_1,\cdots,b_m}]=\cfrac{1}{a_1+\cfrac{1}{\ddots +\cfrac{1}{a_n+\cfrac{1}{b_1+\cfrac{1}{\ddots+\cfrac{1}{b_m+\cfrac{1}{b_1+\ddots}}}}}}}.\]
For any finite word $\omega=(\omega_1,\cdots,\omega_n)\in\{1,2,\cdots\}^n$, denote
\begin{align}\label{def of q_w}
  \frac{p_{\omega}}{q_{\omega}}=[\omega_1,\cdots,\omega_n]=[\omega],
\end{align}
where $p_{\omega}$ and $q_{\omega}$ are coprime and positive.
\begin{theorem}\label{jh}
  The function $\alpha\mapsto E(\alpha)$ is increasing and locally constant almost everywhere. Moreover, the union of all constant intervals is
  \begin{align}\label{g}
    C:=\bigcup_{n=1}^{+\infty}C_1(n)\cup C_2(n)\cup C_3(n),
  \end{align} 
  where
  \[C_1(n)=\bigcup_{\substack{a_{1}\ge 2\\a_2,\cdots,a_{2n}\ge 1}}\bigg[[a_1,\cdots,a_{2n},a_1],[a_1,\cdots,a_{2n},1]\bigg)\]
  \[C_2(n)=\bigcup_{\substack{a_{1}\ge 2\\a_2,\cdots,a_{2n+1}\ge 1}}\bigg[[a_1,\cdots,a_{2n+1},1],[a_1,\cdots,a_{2n+1},a_1]\bigg)\]
  \[C_3(n)=\bigcup_{\substack{a_{2n+1}\ge 2\\a_2,\cdots,a_{2n}\ge 1}}\bigg[[1,a_2,\cdots, a_{2n+1}],[1,\cdots,a_{2n+1}-1,\overline{1,a_2}]\bigg).\]
\end{theorem}
\begin{corollary}\label{elx}
  The function $\alpha\mapsto \dim_H(E(\alpha))$ is increasing and locally constant almost everywhere on $[0,1]$, but it jumps from $0$ to $\frac{1}{2}$ at $0$.
\end{corollary}
\begin{remark}
  We highlight that by Corollary \ref{elx}, we obtain a novel phenomenon. In previous studies, the function $t\mapsto \dim_H(E_f((0,t)))$ invariably forms a devil's staircase--a continuous function that is constant almost everywhere yet surjective. Especially, the continuity is ensured by Raith's theorem \cite[Theorem 2]{R94}. However, in our cases, the function $\alpha\mapsto \dim_H(E(\alpha))$ jumps from $0$ to $\frac{1}{2}$ at the endpoint $0$. This abrupt phase transition arises from the persistent influence of infinitely many branches near $0$, which fundamentally alters the system's dimensional behavior. The proof of \cite[Theorem 2]{R94} critically relies on the finiteness of the number of branches. In fact, by this, the author constructed two bounded linear operators, so that one can use some methods in bounded linear operator theory to prove the continuity of $t\mapsto \dim_H(E_f((0,t)))$.
\end{remark}
Define the \textit{bifurcation set} of the function $\alpha\mapsto E(\alpha)$ as the complement of the set of points where $\alpha\mapsto E(\alpha)$ is locally constant. The following theorem describes the bifurcation set of $\alpha\mapsto E(\alpha)$. 
\begin{theorem}\label{bifurcation set}
  The bifurcation set for the set-valued function $\alpha\mapsto E(\alpha)$ is exactly the set 
  \begin{align}\label{defofe}
    \mathcal{E}:=\left\{x\in [0,1]: T^k(x)\le x,\ \forall k\ge 0\right\}.
  \end{align}
  Moreover, we have $\dim_H(\mathcal{E})=1$.
\end{theorem}
Before giving the Hausdorff dimension of $E(\alpha)$ for all $\alpha\in [0,1]\cap \mathbb{Q}$, we first give some notations.
\par
For $\alpha\in (0,1]$, define 
\[P_n(t,\alpha)=\frac{1}{n}\log \underset{\omega\in\mathcal{D}_{n}(\alpha)}{\sum}q_{\omega}^{-2t},\]
where
\begin{align}\label{defdn}
  \mathcal{D}_{n}(\alpha):=\left\{(a_1(x),\cdots,a_n(x)): x\in E(\alpha)\right\}.
\end{align}
By \cite[pp. 6--7]{MU03}, the following limit exists
\[P(t,\alpha):=\lim\limits_{n\to +\infty}P_n(t,\alpha).\]

\begin{theorem}\label{wsgs}
  For any $\alpha\in [0,1]$, we have the following Hausdorff dimension formula of $E(\alpha)$.
  \begin{enumerate}[1)]
    \item We have $\dim_H(E(0))=0$.
    \item Suppose $\alpha\in (0,1]\cap \mathbb{Q}$. Then, there exists a unique $t>0$, such that $P(t,\alpha)=0$, and we have
    \[\dim_H(E(\alpha))=\inf \{t\ge 0: P(t,\alpha)\le 0\}=\{t>0:P(t,\alpha)=0\}.\]
    \item Suppose $\alpha=[a_1,a_2,\cdots]\notin\mathbb{Q}$. Let $\alpha_{2j}=[a_1,\cdots,a_{2j}]$ and $\alpha_{2j+1}=[a_1,\cdots,a_{2j+1}]$. Then,
    \[\dim_H(E(\alpha))=\lim\limits_{j\to +\infty}\dim_H(E(\alpha_{2j}))=\lim\limits_{j\to +\infty}\dim_H(E(\alpha_{2j+1})).\]
  \end{enumerate}
\end{theorem}
\begin{figure}[htbp]
    \centering
    \includegraphics[width=0.95\linewidth]
        {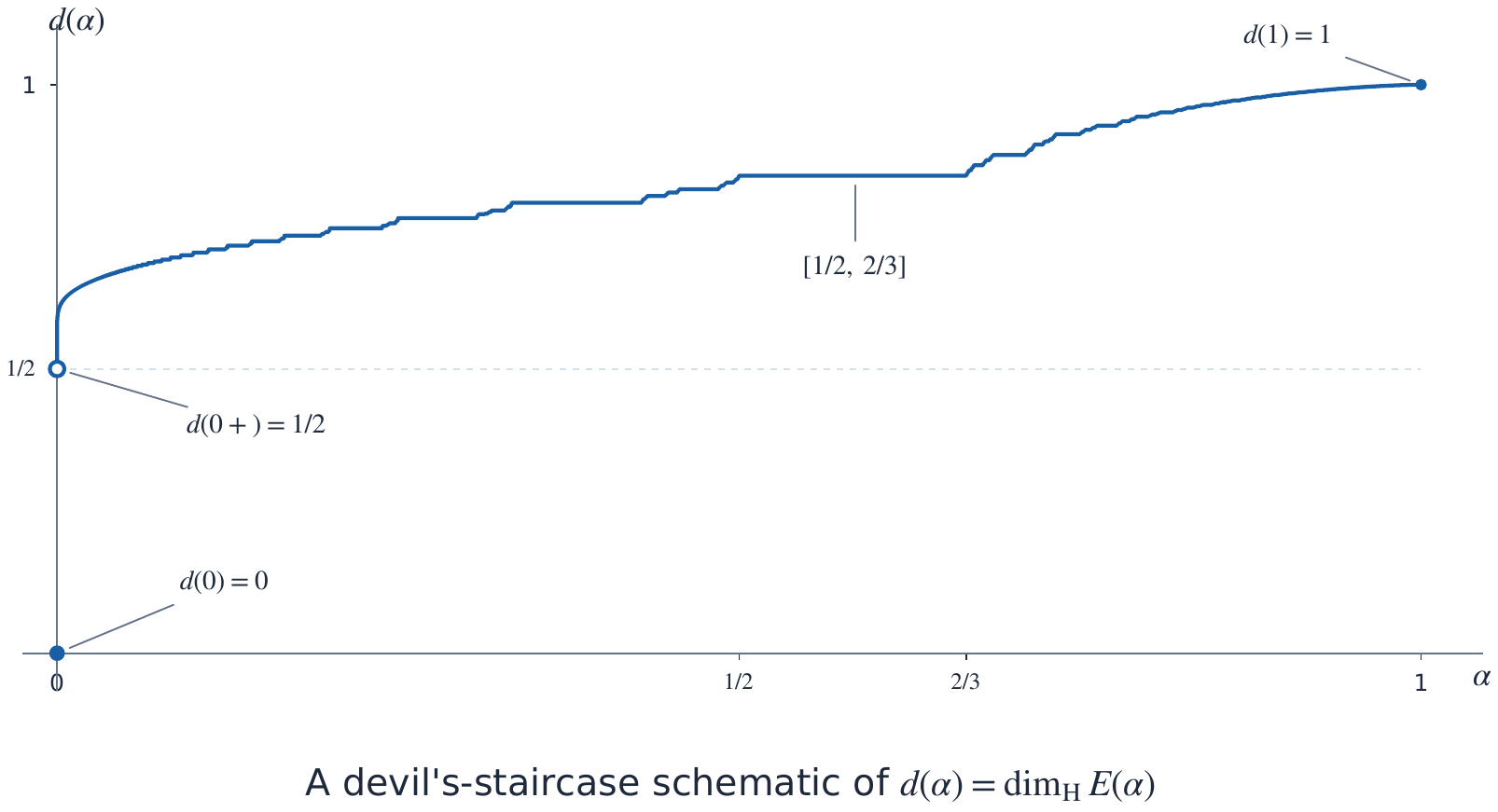}
\end{figure}
We also describe the continuity of the function $\alpha\mapsto \dim_H(E(\alpha))$. 
\begin{theorem}\label{lx}
  The function $\alpha\mapsto \dim_H(E(\alpha))$ is continuous on $(0,1]$, but not continuous at $0$. Moreover, the function $\alpha\mapsto \dim_H(E(\alpha))$ forms a devil's staircase on $(0,1]$.
\end{theorem}
We now give a generalization of Theorem \ref{lx}, whose proof will be omitted since it is similar to the proof of Theorem \ref{lx}.
Let $1=b_0>b_1>b_2>\cdots\longrightarrow 0$. Let $I_n=(b_n,b_{n-1})$. Suppose that each inverse branch $\phi_n:[0,1]\longrightarrow[b_n,b_{n-1}]$ is a $C^2$ monotonic diffeomorphism, with the direction of monotonicity allowed to depend on $n$. Define $T=\phi_n^{-1}$ on $I_n$, and set $T=0$ at all endpoints of the partition. For a finite word $w=(i_1,\ldots,i_m)$, write $\phi_w=\phi_{i_1}\circ\cdots\circ\phi_{i_m}$.
Suppose there exist $C>0$, $0<r<1$, and $K<\infty$ such that
\begin{align}\label{eq:contraction}
 \|\phi_w'\|_\infty&\le Cr^{|w|}
\end{align}
and
\begin{align}\label{eq:distortion}
  \sup_n\left\|\frac{\phi_n''}{\phi_n'}\right\|_\infty&\le K.
\end{align}
Condition \eqref{eq:contraction} is called uniform contraction under iteration, while \eqref{eq:distortion} is called a uniform distortion control.
Let $\ell_n=b_{n-1}-b_n$ and define the critical exponent
\begin{equation}\label{eq:theta}
 \theta=
 \inf\left\{s>0:\sum_n\|\phi_n'\|_\infty^s<\infty\right\}
 =\inf\left\{s>0:\sum_n\ell_n^s<\infty\right\}\in[0,1].
\end{equation}
The two expressions coincide because the distortion control gives
$\|\phi_n'\|_\infty\asymp\ell_n$, with comparison constants independent of $n$.
\begin{theorem}\label{thm:general}
Under the above conditions, the function $\alpha\mapsto \dim_HE(\alpha)$ is continuous on $(0,1]$, and
\[
 \dim_HE(\alpha)\ge\theta\quad(\alpha>0),
 \qquad \dim_HE(0)=0,
 \qquad \lim_{\alpha\downarrow0}\dim_HE(\alpha)=\theta.
\]
\end{theorem}
\begin{remark}
  We will give an example in Section $5$ to show that piecewise monotonicity together with surjectivity does not guarantee Theorem \ref{thm:general}.
\end{remark}

\subsection{Points sharing identical representation under different expansions}
Any real number $x\in\mathbb{R}$ can have different expansions, such as $\beta$-expansion, regular continued fraction expansion, L\"uroth expansion and so on. A compelling question in number theory is to determine the cardinality or Lebesgue measure of numbers sharing identical expansions under two distinct representation systems. In 1970s, Galambos \cite{G71,G76} proved that the set of $x$ whose Engel expansion and Sylvester expansion coincide has zero Lebesgue measure. In 2003, Wu \cite{W03} proved that such a set has Hausdorff dimension $\frac{1}{2}$. %Subsequently, Zhang \cite{Z15} investigated exceptional sets associated with Sylvester expansions, establishing full Hausdorff dimension results under certain growth conditions and identifying divisibility conditions under which the Engel and Sylvester expansions coincide.
\par
We investigate the numbers whose $\alpha$-continued fraction expansion coincides with their regular continued fraction expansion. The $\alpha$-continued fraction expansion is a parameterized generalization of the classical continued fraction. 
\par
Let $\alpha\in (0,1]$. Define the map $T_{\alpha}: [\alpha-1,\alpha]\rightarrow [\alpha-1,\alpha]$ by
\[T_{\alpha}(x)=
\begin{cases}
  \left|\frac{1}{x}\right|-\left\lfloor \left|\frac{1}{x}\right|+1-\alpha\right\rfloor & ,\ x\neq 0, \\
  0 &,\ x=0.
\end{cases}
\]
For any $x\in\mathbb{R}$, set $b_0\in \mathbb{Z}$, such that $x-b_0\in [\alpha-1,\alpha)$. Let $n\ge 1$ be such that $T_{\alpha}^{n-1}(x-b_0)\neq 0$. Then, we define $\varepsilon_n(x)=\text{sgn}\left(T^{n-1}_{\alpha}(x-b_0)\right)$ and
\begin{align}\label{fh}
  b_n(x)=\left\lfloor \left|\frac{1}{T^{n-1}_{\alpha}(x-b_0)}\right|+1-\alpha\right\rfloor.
\end{align}
If there exists $n_0\ge 1$ such that $T^{n_0}_{\alpha}(x-b_0)=0$, then we define $\varepsilon_n(x)=0,\ b_n(x)=0,\ \text{for all}\ n\ge n_0+1$. By definition, we have
\begin{align}\label{a}
  x=b_0(x)+\frac{\varepsilon_1(x)}{b_1(x)+T_{\alpha}(x-b_0)}.
\end{align}
Iterating \eqref{a} gives
\begin{align}\label{b}
  x=b_0(x)+\cfrac{\varepsilon_1(x)}{b_1(x)+\cfrac{\varepsilon_2(x)}{b_2(x)+\ddots}}.
\end{align}
We call \eqref{b} the \textit{$\alpha$-continued fraction} of $x$ and denote it as $[b_0;\varepsilon_1b_1,\varepsilon_2b_2,\cdots]_{\alpha}$. When $\alpha=1$, the expansion \eqref{b} is the classical regular continued fraction. 
\begin{definition}
  Let $x\in \mathbb{R}\setminus \mathbb{Q}$ have the regular continued fraction of $x$ be $[a_0; a_1, a_2,\cdots]$. We say that $x$ has the same regular continued fraction and $\alpha$-continued fraction if $a_n=b_n,\ \text{for all}\  n\ge 0$ and $\varepsilon_n=1,\ \text{for all}\ n\ge 1$. 
\end{definition}
Denote 
  \[S(\alpha)=\left\{x\in \mathbb{R}\setminus \mathbb{Q}: x\ \text{has the same regular continued fraction and $\alpha$-continued fraction} \right\}.\]
For any $A\subset [0,1]$, denote the interior of $A$ by $A^{\circ}$.
The following theorem shows that the set-valued function $\alpha\mapsto S(\alpha)$ is locally constant almost everywhere.  
\begin{theorem}\label{sejh}
  The union of all constant intervals of the function $\alpha\mapsto S(\alpha)$ is
 \[\bigcup_{n=1}^{+\infty}\left(C_1(n)\right)^{\circ}\cup \left(C_2(n)\right)^{\circ}\cup \left(C_3(n)\right)^{\circ},\] 
 where the definition of $C_i(n)$ are the same as that in Theorem \ref{jh}.   
\end{theorem}
The following theorem shows that the Hausdorff dimension of $S(\alpha)$ is the same as that of $E(\alpha)$ defined in \eqref{edy}.
\begin{theorem}\label{sews}
  For all $\alpha\in (0,1)$, we have $\dim_H(S(\alpha))=\dim_H(E(\alpha))$.
\end{theorem}
\section{Preliminaries}
In this section, we first give a brief introduction to regular continued fractions and graph-directed Markov systems. Then, we precent some lemmas which will be used in the proofs of our theorems.
\subsection{Regular continued fraction}
Let $\alpha=1$ be in the definition of $T_{\alpha}$ in Section 1.2. Then, we get the Gauss map $T=T_1:[0,1]\rightarrow [0,1]$,
\[T(x)=
\begin{cases}
  \frac{1}{x}-\left\lfloor \frac{1}{x}\right\rfloor & ,\ x\neq 0, \\
  0 &,\ x=0.
\end{cases}
\]
For any $x\in \mathbb{R}$, set $a_0\in\mathbb{Z}$, such that $x-a_0\in [0,1)$. Note that if $x\notin \mathbb{Q}$, then $T^{n}(x-a_0)\neq 0$ for all $n\ge 1$, and if $x\in \mathbb{Q}$, then there exists some $n_0\ge 0$ such that $T^{n}(x-a_0)=0$ for all $n\ge n_0$. Since $\alpha=1$, we have $\varepsilon_n(x)=1$ if $T^n(x-a_0)\neq 0$. Hence, by \eqref{fh}, we have 
\[a_n(x)=\left\lfloor \frac{1}{T^{n-1}(x-a_0)}\right\rfloor.\]
Then, by \eqref{b}, for any irrational number $x$, we have $x=[a_0(x);a_1(x),a_2(x),\cdots]$ which is the unique regular continued fraction expansion of $x$. For any rational number $x\in (0,1)$, we have $x=[a_0(x);a_1(x),\cdots,a_{n_0}(x)]=[a_0(x);a_1(x),\cdots,a_{n_0}(x)-1,1]$ which are the only two regular continued fraction expansions of $x$.
\par
We will provide a symbolic description of the continued fraction expansions in terms of symbolic strings. Let us start with some notation. If $S = (s_1,\cdots,s_n)$ is a finite string of positive integers and $x\in [0,1]$ is a real number, we denote
$[0;S]=[0;s_1,\cdots,s_n]$,
and
\[[0;S,x]=\cfrac{1}{s_1\cfrac{1}{\ddots +\cfrac{1}{s_n+x}}}.\] 
We will also introduce a total ordering on the space of finite strings of given length. 
%Given two distinct finite strings $S=(s_1,\cdots,s_n)$ and $T=(t_1,\cdots,t_n)$ of equal length $n$, let $\ell := \min\{1\le i\le n-1:s_i \neq t_i\}$. When $\ell$ is even, we define $S<T$ if $s_{\ell}<t_{\ell}$. When $\ell$ is odd, we define $S<T$ if $s_{\ell}>t_{\ell}$. It can be checked that the same definition gives a total ordering on the space of infinite strings. Further, we have $S<T$ if and only if $[0;S]<[0;T]$. Hence, this ordering can be obtained by pulling back the order structure on $\mathbb{R}$, via identification of a string with the value of the corresponding regular continued fraction expansion.
The following lemma shows how to compare the size of two real numbers $x$ and $y$ by their regular continued fraction expansions.
\begin{lemma}
  Let $x=[x_1,x_2,\cdots]$ and $y=[y_1,y_2,\cdots]$ be two different real numbers in $(0,1]$. 
  \begin{enumerate}[1)]
    \item Suppose $x,y\notin \mathbb{Q}$ and $x\neq y$. Define $\ell:=\min\{i\ge 1:x_i \neq y_i\}$. If $\ell$ is even, then $x>y$ if and only if $x_{\ell}>y_{\ell}$. If $\ell$ is odd, then $x>y$ if and only if $x_{\ell}<y_{\ell}$.
    \item Suppose $x=[x_1,\cdots,x_n]\in\mathbb{Q}$ and $y\notin\mathbb{Q}$ or $y=[y_1,\cdots,y_m]$ with $n\le m$. We distinguish two subcases.
    \begin{enumerate}
      \item Suppose $y_i=x_i$ for all $1\le i\le n$. Then $x>y$ if and only if $n$ is odd.
      \item Suppose there exists some $i$ such that $x_i\neq y_i$. Define $\ell:=\min\{1\le i\le n:x_i \neq y_i\}$. If $\ell$ is even, then $x>y$ if and only if $x_{\ell}>y_{\ell}$. If $\ell$ is odd, then $x>y$ if and only if $x_{\ell}<y_{\ell}$.
    \end{enumerate}
  \end{enumerate}
\end{lemma}
\subsection{Graph directed Markov systems}
Consider a $0$--$1$ incidence matrix $A:I\times I \rightarrow \{0, 1\}$ where $I$ is a countable alphabet.
Let
\[E_{A}^{\infty}= \{\omega\in I^{\infty}: A_{\omega_i\omega_{i+1}}=1,\ \forall\ i \ge 1\}\]
be the space of all $A$-admissible infinite sequences with terms in $I$, i.e., we are considering a directed graph with vertices being the elements of $I$ and a directed edge from $\omega_i$ to $\omega_j$ when $A_{\omega_i\omega_j} = 1$ but not otherwise. Let $E_{A}^n$ be the collection of all subwords of $E_{A}^{\infty}$ of length $n$, and let $E_{A}^*$ be the collection of all finite subwords of $E_{A}^{\infty}$. The incidence matrix $A$ is said to be \textit{finitely primitive} if there exists
a finite set $\Lambda\subset E_{A}^*$ consisting of words of the same length such that for all $\omega_i,\omega_j\in I$ there exists a path $\gamma\in \Lambda$ for which $\omega_i\gamma\omega_j\in E_{A}^{\infty}$. Here $\omega_i\gamma\omega_j$ is the concatenation of the three words.
\par
For any $t\ge 0$, let
\[P_n(t)=\frac{1}{n}\log \underset{\omega\in E_{A}^n}{\sum}q_{\omega}^{-2t},\]
where $q_\omega$ is defined by \eqref{def of q_w}.  
Let $E(A):=\{[\omega]:\omega\in E_A^{\infty}\}$.
The following two lemmas shows that the Hausdorff dimension of $E(A)$ can be described by using the function $t\mapsto P_n(t)$.
\begin{lemma}\cite[Proposition 4.2.8]{MU03}\label{ylhs}
  Suppose $A$ is a finitely primitive $0$-$1$ incidence matrix. Then, we have
  \begin{enumerate}[1)]
    \item For all $t\ge 0$, the limit $\lim\limits_{n\to +\infty}P_n(t)=P(t)\in [-\infty,+\infty]$,
    \item The function $t\mapsto P(t)$ is convex, continuous, and non-increasing on $[0,+\infty)$, strictly decreasing on $[\theta, +\infty)$ to negative infinity, where $\theta:=\inf \{t\ge 0: P(t)< +\infty\}$.
  \end{enumerate}
\end{lemma}
\begin{lemma}\cite[Theorem 4.2.13]{MU03}\label{wsdl}
  Suppose $A$ is a finitely primitive $0$-$1$ incidence matrix. Then, we have 
  \[\dim_H(E(A))=\inf \{t\ge 0: P(t)\le 0\}=\min \{t\ge 0: P(t)\le 0\}.\]
\end{lemma}
\section{Structure of $E(\alpha)$ and proofs of Theorems \ref{jh} and \ref{bifurcation set}}
In this section, we will use the methods of Carminati and Tiozzo in \cite{CT22} to describe the structure of $E(\alpha)$ and $\mathcal{E}$. By studying $\mathcal{D}_n(\alpha)$ defined in \eqref{defdn}, we give a proof of Theorem \ref{jh}. By finding a subset of $\mathcal{E}$ which has full Hausdorff dimension, we give a proof of Theorem \ref{bifurcation set}.
\begin{lemma}\label{lembj}
  For any fixed $k\ge 0$, the set
  \[A_k(a,b)=\left\{x\in [0,1]:T^k(x)=0\right\}\bigcup\left\{x\in [0,1]:T^{k+1}(x)\in [a,b]\right\}\]
  is closed, where $[a,b]$ is any closed subinterval of $[0,1)$. 
\end{lemma}
\begin{proof}
  We prove Lemma \ref{lembj} by induction on $k$. 
  \par
  Suppose $k=0$. Then, 
  \[A_0(a,b)=\{0\}\cup\bigcup_{n=1}^{+\infty}\left[\frac{1}{b+n},\frac{1}{a+n}\right].\]
  Let $\{x_n\}$ be a sequence in $A_0(a,b)$ such that $\lim_{n\to\infty}x_n=x$. Now, we prove that $x\in A_0(a,b)$. Since $0\in A_0(a,b)$, without loss of generality, we suppose $x>0$. If for any $N\in\mathbb{N}$, there exists some $n$, such that $x_n\in \left[\frac{1}{b+N},\frac{1}{a+N}\right]$, then $\lim_{n\to\infty}x_n=0$, which leads to a contradiction. Hence, there exists some $N$, such that $x_n\in \left[\frac{1}{b+N},\frac{1}{a+N}\right]$ for all $n$ large enough. Since $b-a<1$, then $\left[\frac{1}{b+n_1},\frac{1}{a+n_1}\right]\cap \left[\frac{1}{b+n_2},\frac{1}{a+n_2}\right]=\emptyset$ for any $n_1\neq n_2$. Hence, we have $x\in \left[\frac{1}{b+N},\frac{1}{a+N}\right]$, which implies that $A_0(a,b)$ is closed.
  \par
  Suppose the conclusion holds for $k=\ell$. Now we consider $k=\ell+1$. Let $\{x_n\}$ be a sequence in $A_{\ell+1}(a,b)$ with $\lim_{n\to\infty}x_n=x$. Since
  \[A_{\ell+1}(a,b)=\left\{x\in [0,1]:T^{\ell+1}(x)=0\right\}\cup\bigcup_{n=1}^{+\infty}\left\{x\in [0,1]:T^{\ell+1}(x)\in \left[\frac{1}{b+n},\frac{1}{a+n}\right]\right\},\]
  we can suppose $T^{\ell+1}(x)>0$, otherwise $x\in A_{\ell+1}(a,b)$ obviously. Since $T^{\ell+1}(x)>0$, then $x=[a_1,\cdots, a_{\ell+1},a_{\ell+2},\cdots]$ with $a_{\ell+2}\ge 1$. Let $x_n=\left[a_1^{(n)},\cdots\right]$. Since $\lim_{n\to\infty}x_n=x$, then for $n$ large enough, we have $a_i^{(n)}=a_i$ for all $1\le i\le \ell+1$ and $|a_{\ell+2}^{(n)}-a_{\ell+2}|\le 1$. Hence, $T^{\ell+1}(x_n)\ge \frac{1}{a^{(n)}_{\ell+2}+1}\ge \frac{1}{a_{\ell+2}+2}$ when $n$ is large enough. Hence, for any $n$ large enough, the sequence $\{T^{\ell+1}(x_n)\}$ is bounded away from zero. Since $\{x_n\}$ is a sequence in $A_{\ell+1}(a,b)$, then there exists $N\ge 1$, such that
  \[x_n\in \bigcup_{m=1}^{N}\left\{x\in [0,1]:T^{\ell+1}(x)\in \left[\frac{1}{b+m},\frac{1}{a+m}\right]\right\}\]
  when $n$ is large enough. Note that
  \[\left\{x:T^{\ell+1}(x)\in \left[\frac{1}{b+m},\frac{1}{a+m}\right]\right\}=\bigcup_{q=1}^{+\infty}\left\{x:T^{\ell}(x)\in \left[\cfrac{1}{\cfrac{1}{a+m}+q},\cfrac{1}{\cfrac{1}{b+m}+q}\right]\right\}.\]
  Then, for $n$ large enough,
  \[x_n\in\bigcup_{m=1}^{N}\bigcup_{q=1}^{+\infty}\left\{x\in [0,1]:T^{\ell}(x)\in \left[\cfrac{1}{\cfrac{1}{a+m}+q},\cfrac{1}{\cfrac{1}{b+m}+q}\right]\right\}\subset \bigcup_{m=1}^{N}A_{\ell}\left(\frac{1}{b+m},\frac{1}{a+m}\right).\]
  By induction, each set $A_{\ell}\left(\frac{1}{b+m},\frac{1}{a+m}\right)$ is closed. Hence, for $n$ large enough, $\{x_n\}$ is a sequence in a closed set, which implies that 
  \[x\in \bigcup_{m=1}^{N}A_{\ell}\left(\frac{1}{b+m},\frac{1}{a+m}\right).\]
  Hence, there exists $m_0$, such that $T^{\ell+1}(x)\in \left(\frac{1}{b+m_0},\frac{1}{a+m_0}\right)$, which completes the proof.
\end{proof}
\begin{proposition}\label{propbj}
  For any $\alpha\in [0,1]$, the set $E(\alpha)$ is compact.
\end{proposition}
\begin{proof}
  Since $E(0)=\{0\}$ and $E(1)=[0,1]$, without loss of generality, we can suppose that $\alpha\in (0,1)$.
  By definition,
  \[E(\alpha)=[0,\alpha]\cap\left(\bigcap_{k=1}^{+\infty}\left\{x\in [0,1]: T^k(x)\in [0,\alpha]\right\}\right).\]
  For any $k\ge 1$, since $T^{k-1}(x)=0$ implies that $T^k(x)=0$, then
  \[\left\{x: T^k(x)\in [0,\alpha]\right\}=\left\{x:T^{k-1}(x)=0\right\}\cup\left\{x: T^k(x)\in [0,\alpha]\right\}=A_{k-1}(0,\alpha),\]
  where $A_{k-1}(0,\alpha)$ is defined as in Lemma \ref{lembj}.
  Hence, by Lemma \ref{lembj}, the set $E(\alpha)$ is a countable intersection of closed sets, which completes the proof.
\end{proof}
  Now, let $t(\alpha)=\max E(\alpha)$. By using the function $\alpha\mapsto t(\alpha)$, we give a description of intervals on which the set $E(\alpha)$ does not change. Recall that $\mathcal{E}=\{x\in [0,1]: T^k(x)\le x,\ \forall k\ge 0\}$.
\begin{proposition}\label{propbb1}
  If $(a,b)$ is a connected component of $[0,1]\setminus \mathcal{E}$, then $E(\alpha)=E(a),\ \forall \alpha\in [a,b)$.
\end{proposition}
\begin{proof}
  First we claim that $t(\alpha)\in \mathcal{E}\cap [0,\alpha]$. Indeed, since $t(\alpha)\in E(\alpha)$, then $T^k(t(\alpha))\in E(\alpha),\ \forall k\ge 0$. Hence, by the definition of $t(\alpha)$, we have $T^k(t(\alpha))\le t(\alpha),\ \forall k\ge 0$, which implies $t(\alpha)\in\mathcal{E}$. Since $E(\alpha)\subset [0,\alpha]$, we prove the claim.
  \par
  Now, fix an $\alpha\in (a,b)$. Since $(a,b)$ is a connected component of $[0,1]\setminus \mathcal{E}$, then $a\in \mathcal{E}$. Hence, we have $T^k(a)\le a,\ \forall k\ge 0$. Thus, $a\in E(a)$. Since $a<\alpha$, then $a\in E(\alpha)$. By the definition of $t(\alpha)$, we have $a\le t(\alpha)$. On the other hand, since $(a,b)\cap \mathcal{E}=\emptyset$ and $t(\alpha)\in \mathcal{E}\cap [0,\alpha]$, then $t(\alpha)\le a$. Therefore, we have $t(\alpha)=a$. Since $E(\alpha)$ is increasing, then for any $a<\alpha_1<\alpha_2<b$, we have $E(\alpha_1)\subset E(\alpha_2)$. If $E(\alpha_2)\setminus E(\alpha_1)\neq \emptyset$, then $t(\alpha_2)\ge t(\alpha_1)$, which leads to a contradiction.
\end{proof}
\begin{proposition}\label{aejc}
  The function $\alpha\mapsto E(\alpha)$ is locally constant almost everywhere and increasing.
\end{proposition}
\begin{proof}
  By the definition of $E(\alpha)$, it is clear that $E(\alpha)$ is increasing. We write $\mathcal{E}$ as a countable union, 
  \[\mathcal{E}=\bigcup_{n=1}^{+\infty}\left(\left\{x\in[0,1]: T^k(x)\le x,\ \forall k\ge 0\right\}\cap \left[0,1-\frac{1}{n}\right]\right)\cup \{1\}.\]
  For any $x\le 1-\frac{1}{n}$, we have $T^k(x)\le x\le 1-\frac{1}{n}$. Hence,
  \[\mathcal{E}\subset \bigcup_{n=1}^{+\infty}E\left(1-\frac{1}{n}\right)\cup \{1\}.\]
  Since $E\left(1-\frac{1}{n}\right)\subset \left[0,1-\frac{1}{n}\right]$ is a $T$-invariant set, then $E\left(1-\frac{1}{n}\right)$ has zero Lebesgue measure for all $n\ge 1$. Hence, $\mathcal{E}$ has zero Lebesgue measure. By Proposition \ref{propbb1}, the function $\alpha\mapsto E(\alpha)$ is locally constant almost everywhere.
\end{proof}
Proposition \ref{aejc} shows that the function $\alpha\mapsto E(\alpha)$ is locally constant almost everywhere, but it cannot give any information about the interval of $\alpha$ that $E(\alpha)$ does not change. In order to prove Theorem \ref{jh}, we need the following lemma.\par
Recall that $\mathcal{D}_n(\alpha)$ is the collection of all admissible strings of length $n$ over the set $E(\alpha)$, i.e., $\mathcal{D}_{n}(\alpha)=\left\{(a_1(x),\cdots,a_n(x)): x\in E(\alpha)\right\}$ and \[\mathcal{G}_n=\left\{(a_1,\cdots,a_n)\in\mathbb{N}^n:\ \exists\ k\le n\ \text{such that}\ a_k\le a_1-1\right\}.\]
\begin{lemma}\label{zf}
  Let $\alpha=[a_1,\cdots,a_n]\in\mathbb{Q}$ with $a_n\ge 2$. 
  \begin{enumerate}[1)]
    \item Suppose $n$ is odd. Let $\alpha_j=[a_1,\cdots,a_n-1,1,j]$. Then, we have $\mathcal{D}_{n+2}(\alpha_j)\subset\mathcal{D}_{n+2}(\alpha)\cup\mathcal{F}_{n+2}(\alpha_j)$, where
    \[\mathcal{F}_{n+2}(\alpha_j)=
      \begin{cases}
        \emptyset ,\quad \text{if}\ (a_1,\cdots,a_n-1,1)\in \mathcal{G}_{n+1}\ \text{or}\ a_1=1\ \text{and}\ a_2\le j-1,\\
        (t,a_1,\cdots,a_n-1,1)_{t\ge a_1}\cup (a_1,\cdots,a_n-1,1,s)_{j\le s\le a_2},\quad \text{otherwise}.
      \end{cases}\]
    \item Suppose $n$ is even. Let $\alpha_j=[a_1,\cdots,a_n,j]$. Then, we have $\mathcal{D}_{n+1}(\alpha_j)\subset\mathcal{D}_{n+1}(\alpha)\cup\mathcal{F}_{n+1}(\alpha_j)$, where
    \[\mathcal{F}_{n+1}(\alpha_j)=
      \begin{cases}
        \emptyset ,\quad \text{if}\ (a_1,\cdots,a_n)\in \mathcal{G}_{n}\ \text{or}\ a_1=a_n\ \text{and}\ a_2\le j-1,\\
        (t,a_1,\cdots,a_n)_{t\ge a_1}\cup (a_1,\cdots,a_n,s)_{s\ge j},\quad \text{otherwise}.
      \end{cases}\]
  \end{enumerate}
\end{lemma}
\begin{proof}
  We only prove $1)$, the proof of $2)$ is the same as the proof of $1)$. Since $n$ is odd, we have $\alpha_j>\alpha$. Hence, 
  \[E(\alpha_j)=E(\alpha)\cup\bigcup_{k=0}^{+\infty}\left\{x\in E(\alpha_j):T^{k}(x)\in (\alpha,\alpha_j]\right\}.\]
  Now, we consider the set 
  \begin{align}\label{chaji1}
    \bigcup_{k=0}^{+\infty}\left\{x\in E(\alpha_j):T^{k}(x)\in (\alpha,\alpha_j]\right\}.
  \end{align}
  We first prove that \eqref{chaji1} is an empty set if $(a_1,\cdots,a_n-1,1)\in \mathcal{G}_{n+1}\ \text{or}\ a_1=1\ \text{and}\ a_2\le j-1$ by contradiction.
  Suppose \eqref{chaji1} is not empty. Then, there exists some $k\ge 0$, such that $T^k(x)\in (\alpha,\alpha_j]$, then by comparing the expansion of $\alpha$ and $\alpha_j$, we have $T^k(x)=[a_1,\cdots,a_n-1,1,\cdots]$. 
  \par
  Suppose $(a_1,\cdots,a_n-1,1)\in\mathcal{G}_{n+1}$. Then there exists some $n_0\ge 2$, such that $a_{n_0}< a_1$. Hence, $T^{k+n_0-1}(x)=[a_{n_0},\cdots]>\alpha_j$. This leads to a contradiction.
  \par
   Suppose $a_1=1$ and $a_2\le j-1$. Since $T^k(x)<\alpha_j$, then $T^{k}(x)=[1,a_2,\cdots,a_n-1,1,s,\cdots]$ with $s\ge j$. Therefore, we have $T^{k+n}(x)=[1,s,\cdots]>\alpha_j$. This leads to a contradiction. Hence, under these two cases, we have $E(\alpha_j)=E(\alpha)$, which implies $\mathcal{D}_{n+2}(\alpha_j)=\mathcal{D}_{n+2}(\alpha)$. 
  \par
  Now, we consider other cases. It can be checked that $\mathcal{D}_{n+1}(\alpha_j)\subset\mathcal{D}_{n+1}(\alpha)\cup \{(a_1,\cdots,a_n-1,1)\}$.
  Hence, we need only consider the following two kinds of strings $(t,a_1,\cdots,a_n-1,1)_{t\ge 1}$ and $(a_1,\cdots,a_n-1,1,s)_{s\ge 1}$. Now, we prove that 
  \[(t,a_1,\cdots,a_n-1,1)_{t\le a_1-1},\ (a_1,\cdots,a_n-1,1,s)_{s\le j-1},\ (a_1,\cdots,a_n-1,1,s)_{s\ge a_2+1}\notin \mathcal{D}_{n+2}(\alpha_j)\]
  by contradiction. Suppose there exists some $t\le a_1-1$ such that $(t,a_1,\cdots,a_n-1,1)\in \mathcal{D}_{n+2}(\alpha_j)$. Then, there exists some $k_0$, such that $T^{k_0}(x)=[t,a_1,\cdots]>\alpha_j$. Suppose there exists some $s\le j-1$ such that $(a_1,\cdots,a_n-1,1,s)\in \mathcal{D}_{n+2}(\alpha_j)$. Then, there exists some $k_1$, such that $T^{k_1}(x)=[a_1,\cdots, a_n-1,1,s]>\alpha_j$. Finally, suppose there exists some $s\ge a_2+1$ such that $(a_1,\cdots,a_n-1,1,s)\in \mathcal{D}_{n+2}(\alpha_j)$. Then, there exists some $k_2$, such that $T^{k_2}(x)=[1,s,\cdots]>\alpha_j$. The three subcases all lead to a contradiction. Hence, combining the three cases above, we complete the proof of $1)$.
\end{proof}
\begin{corollary}\label{bb2}
  The set $E(\alpha)$ does not change when $\alpha\in \bigg[[a_1,\cdots,a_{2n},a_1],[a_1,\cdots,a_{2n},1]\bigg]$ and $\alpha\in \bigg[[a_1,\cdots,a_{2n+1},1],[a_1,\cdots,a_{2n+1},a_1]\bigg)$. Moreover, if $a_{2n+1}\ge 2$, then the set $E(\alpha)$ also does not change when $\alpha\in \bigg[[1,a_2,\cdots, a_{2n+1}],[1,\cdots,a_{2n+1}-1,\overline{1,a_2}]\bigg)$.
\end{corollary}
\begin{proof}
  First, we prove $E([a_1,\cdots,a_{2n},a_1])=E([a_1,\cdots,a_{2n},1])$ by contradiction. Suppose there exists some $x\in E([a_1,\cdots,a_{2n},1])\setminus E([a_1,\cdots,a_{2n},a_1])$. Then, there exist some $k_0$ such that $T^{k_0}(x)\in \bigg([a_1,\cdots,a_{2n},a_1],[a_1,\cdots,a_{2n},1]\bigg)$. Hence, we have $T^{k_0}(x)=[a_1,\cdots, a_{2n},s,\cdots]$ for some $s\le a_1-1$. Thus, we have $T^{k_0+2n}(x)=[s,\cdots]>[a_1,\cdots,a_{2n},1]$. This leads to a contradiction. By the same arguments above, we also have $E([a_1,\cdots,a_{2n+1},a_1])\subset E([a_1,\cdots,a_{2n+1},1])\cup [a_1,\cdots,a_{2n+1},a_1]$. Hence, The set $E(\alpha)$ does not change when $\alpha\in \bigg[[a_1,\cdots,a_{2n+1},1],[a_1,\cdots,a_{2n+1},a_1]\bigg)$.
  \par
  Now, we prove the moreover part of Corollary \ref{bb2}. Since $a_{2n+1}\ge 2$ and $a_1=1,\ a_2\le a_2+1-1$, by $2)$ of Lemma \ref{zf}, we have $\mathcal{D}_{2n+2}([1,\cdots,a_{2n+1}])=\mathcal{D}_{2n+2}([1,\cdots,a_{2n+1}-1,1,a_2+1])$. Hence, by the definition of $\mathcal{D}_k(\alpha)$, we have $E([1,\cdots,a_{2n+1}])=E([1,\cdots,a_{2n+1}-1,1,a_2+1])$. Since the length of $[1,\cdots,a_{2n+1}-1,1,a_2+1]$ is also odd and $a_2+1\ge 2$, we also have $E([1,\cdots,a_{2n+1}-1,1,a_2+1])=E([1,\cdots,a_{2n+1}-1,1,a_2,1,a_2+1])$. Denote by $(1,a_2)^m$ the string obtained by repeating $(1,a_2)$ $m$ times. Repeating the process above, we have that
  \[E([1,\cdots,a_{2n+1}])=E([1,\cdots,a_{2n+1}-1,(1,a_2)^m,1,a_2+1]),\quad \forall m\ge 0.\] 
  For any $[1,\cdots,a_{2n+1}] \le \alpha<[1,\cdots,a_{2n+1}-1,\overline{1,a_2}]$, there exists some $m_0\ge 0$ such that 
  $\alpha<[1,\cdots,a_{2n+1}-1,(1,a_2)^{m_0},1,a_2+1]$, which completes the proof.
\end{proof}
\begin{proof}[Proof of Theorem \ref{jh}]
  %Suppose $F=\bigcup\limits_{n=1}^{+\infty}F_n$ with
  %\[F_n=\{[x_1,\cdots]\in [0,1]\setminus \mathbb{Q}: x_{2k+1}=1,\ \forall k\ge n\}.\]
   %Since $F_n$ has zero Lebesgue measure for all $n\ge 1$ (see \cite[Theorem 2.1]{HJS12}). Hence, $F$ also has zero Lebesgue measure. 
   By combining Proposition \ref{aejc} and Corollary \ref{bb2}, we need only prove that for any $x\in [0,1]\setminus \mathbb{Q}$, for any $\varepsilon>0$, there exists a point 
  \[y\in \bigcup_{n=1}^{+\infty}C_1(n)\cup C_2(n)\cup C_3(n)\]
  such that $|x-y|<\varepsilon$. 
  We distinguish two cases.
  \begin{enumerate}[1)]
    \item Suppose $x=[x_1,x_2,\cdots]$ with $x_1\ge 2$. Then, for any $\varepsilon>0$, there exists some $n_0\ge 1$ such that $|x-y|<\varepsilon$, where $y=[x_1,\cdots, x_{2n_0+1},1]$. It can be checked that $y\in C_2(n_0)$, which completes the proof of this case.
    \item Suppose $x=[1,x_2,\cdots]$. For all $\varepsilon>0$, there exists some $k_0\ge 1$ such that $|x-y|<\varepsilon$, where $y=[x_1,\cdots, x_{2k_0},2]$. It can be checked that $y\in C_3(k_0)$, which completes the proof of this case.
  \end{enumerate}
\end{proof}
In order to prove theorem \ref{bifurcation set}, we prove the following proposition first.
\begin{proposition}\label{bfca}
  The set $\mathcal{E}=\{x\in [0,1]: T^k(x)\le x,\ \forall k\ge 0\}$ is compact.
\end{proposition}
\begin{proof}
  Let $\{x_n\}$ be a sequence in $\mathcal{E}$ such that $\lim\limits_{n\to +\infty}x_n=x$. Now, we prove that $x\in \mathcal{E}$. Since $0\in\mathcal{E}$, without loss of generality, we suppose $x>0$. We prove $x\in \mathcal{E}$ by contradiction. Suppose $x\notin\mathcal{E}$. Then, there exists some $k_0\in \mathbb{N}$ such that $T^{k_0}(x)>x>0$. Hence,
  \[x=[a_1,\cdots,a_{k_0},a_{k_0+1},\cdots]\quad\text{with}\quad a_{k_0+1}\ge 1.\]
  Let $x_n=[a_1^{(n)},\cdots]$. Since $\lim\limits_{n\to +\infty}x_n=x$, then for $n$ large enough, we have $a_i=a_i^{(n)}$ for all $1\le i\le k_0$. Hence, for $n$ large enough, we have 
  \[T^k(x_n)\in \left(\frac{1}{a_{k+1}+1},\frac{1}{a_{k+1}}\right)\quad\text{for all}\quad 0\le k\le k_0-1.\]
  Since $x\in \left(\frac{1}{a_1+1},\frac{1}{a_1}\right)$ and $x_n\in \left(\frac{1}{a_1+1},\frac{1}{a_1}\right)$ for $n$ large enough, by noting that $T$ is continuous on $\left(\frac{1}{a_1+1},\frac{1}{a_1}\right)$, then 
  \[\lim_{n\to +\infty}T(x_n)=T(x).\]
  Since $T(x)\in \left(\frac{1}{a_2+1},\frac{1}{a_2}\right)$ and $T(x_n)\in \left(\frac{1}{a_2+1},\frac{1}{a_2}\right)$ for $n$ large enough, by noting that $T$ is continuous on $\left(\frac{1}{a_2+1},\frac{1}{a_2}\right)$, then 
  \[\lim_{n\to +\infty}T^2(x_n)=\lim_{n\to +\infty}T(T(x_n))=T(T(x))=T^2(x).\]
  Repeating this process, we have $\lim_{n\to +\infty}T^{k_0}(x_n)=T^{k_0}(x)$. Since $x_n\in \mathcal{E}$, then
  \[x=\lim_{n\to +\infty}x_n\ge \lim_{n\to +\infty}T^{k_0}(x_n)=T^{k_0}(x)>x.\]
  This leads to a contradiction. Hence, the set $\mathcal{E}$ is closed, which completes the proof.
\end{proof}
\begin{proof}[Proof of Theorem \ref{bifurcation set}]
  Denote the bifurcation set for the set-valued function $\alpha\mapsto E(\alpha)$ by $\mathcal{B}$. First, we prove 
  $\mathcal{E}\subset \mathcal{B}$. For any $x_0\in \mathcal{E}\setminus \{0\}$, we have 
  \[E(x_0)=\{x\in [0,1]:  T^k(x)\le x_0,\ \forall k\ge 0\}\supset \{x_0\}\sqcup \{x\in [0,1]:  T^k(x)< x_0,\ \forall k\ge 0\}.\]
  Hence, for any $0<\delta\le x_0$,
  \[E(x_0-\delta)\subsetneqq E(x_0).\]
  By definition of $\mathcal{B}$, we have $x_0\in\mathcal{B}$ for all $x_0>0$. It can be checked that $0\in \mathcal{B}$.
  \par
  Now, we prove $\mathcal{B}\subset \mathcal{E}$. For any $x_0\in [0,1]\setminus \mathcal{E}$, we prove $x_0\in [0,1]\setminus \mathcal{B}$. By Proposition \ref{bfca}, the set $(0,1)\setminus \mathcal{E}$ is open. Since $0, 1\in \mathcal{E}$, then there exist intervals $(a_i,b_i)\subset (0,1)$ such that
  \[[0,1]\setminus \mathcal{E}=(0,1)\setminus \mathcal{E}=\bigcup_{i=1}^{+\infty} (a_i,b_i).\]
  Hence, there exists some $i_0$ such that $x_0\in (a_{i_0},b_{i_0})$. By Proposition \ref{propbb1}, we have $E(\alpha)=E(x_0)$ for all $\alpha\in (a_{i_0},b_{i_0})$. This implies $x_0\in [0,1]\setminus \mathcal{B}$. Hence, we have $\mathcal{B}=\mathcal{E}$.
  \par
  Now, we prove $\dim_H(\mathcal{E})=1$. Denote
  \[\bigcup_{m=2}^{+\infty}\left\{x=[1,m,x_3,\cdots]:x_i\le m-1,\ \forall i\ge 3\right\}\] 
  by $C$. For any $x\in C$, we have $T(x)=[m,x_3,x_4,\cdots]<x$ and $T^k(x)=[x_{k+1},\cdots]$ for any $k\ge 2$. If $x_{k+1}\ge 2$, then $T^k(x)<x$. If $x_{k+1}=1$, then by noting that $x_{k+2}\le m-1$, we have $T^k(x)<x$. Hence, we have 
  $C\subset \mathcal{E}$. Since 
  \[\dim_H(C)=\underset{m\ge 2}{\sup}\dim_H\left(\left\{x=[1,m,x_3,\cdots]:x_i\le m-1,\ \forall i\ge 3\right\}\right)=1,\]
  then $\dim_H(\mathcal{E})=1$.
\end{proof}
\section{Hausdorff \text{dim}ension of $E(\alpha)$ for $\alpha\in \mathbb{Q}$}
In this section, we describe the Hausdorff dimension of $E(\alpha)$ for all $\alpha\in [0,1]\cap \mathbb{Q}$.
\begin{proposition}\label{wsgs1}
  For any $\alpha\in\mathbb{Q}\setminus \{0\}$, we have 
  \begin{align}\label{wsgs2}
    \dim_H(E(\alpha))=\inf\left\{t\ge 0: P(t,\alpha)\le 0\right\},
  \end{align}
  where
  \[P(t,\alpha)=\lim_{n\to +\infty}\frac{1}{n}\log \underset{\omega\in \mathcal{D}_n(\alpha)}{\sum}q_{\omega}^{-2t}.\]
  For $\alpha=0$, we have $\dim_H(E(\alpha))=0$.
\end{proposition}
\begin{proof}
  Let $N\ge 1$ be a positive integer.
  Note that
  \[E\left(\frac{1}{N}\right)=\left\{x=[x_1,\cdots]: x_i\ge N,\ \forall i\ge 1 \right\}.\]
  By \cite{G41}, the equation \eqref{wsgs2} holds for $\alpha=\frac{1}{N}$. Hence, with out loss of generality, we can suppose $\alpha=[N,a_2,\cdots]$. Now, we distinguish two cases.
  \begin{enumerate}[i)]
    \item Suppose $\alpha=[N,a_2,\cdots,a_n]$ with $a_n\ge 2$ and $a_2\ge N+1$. By the definition of $\mathcal{D}_{n}(\alpha)$, we have
    \[E(\alpha)=\left\{x\in [0,1]: (a_k(x),a_{k+1}(x),\cdots, a_{k+n-1}(x))\in\mathcal{D}_{n}(\alpha),\ \forall k\ge 1\right\}.\]
    Now, we use the theories of Markov systems (see \cite{MU03}) to prove \eqref{wsgs2}. First, we recast the set $E(\alpha)$ in the language of graph-directed Markov system, such that we can use Lemmas \ref{ylhs} and \ref{wsdl}. Let $I=\{(a_1,\cdots,a_{n-1}): (a_1,\cdots,a_{n-1})\in\mathcal{D}_{n-1}(\alpha)\}.$ We also denote $I$ by $\{\omega_i\}_{i\ge 1}$ for convenience.
    Define the infinite incidence matrix $A$ by $a_{ij}=1$ if $\omega_i\omega_j\in\mathcal{D}_{2n-2}(\alpha)$ and $a_{ij}=0$ otherwise.
    \par
    Now, we prove that $A$ is finitely primitive by induction on the length $n$. Suppose $n=1$. Then $a_{ij}=1$ for all $i,j\ge 1$. Hence, the matrix $A$ is finitely primitive obviously. Suppose $n=2$. Since $a_2\ge N+1$, then for any $\omega_i,\ \omega_j\in I$, we have $[\omega_i,N+1,\cdots]\le \alpha$ and $\ [N+1, \omega_j\cdots]\le \alpha$. Hence, we have $\omega_i(N+1)\omega_j\in\mathcal{D}_3(\alpha)$ for all $\omega_i,\omega_j\in I$. Hence, the matrix $A$ is finitely primitive. Suppose $n=3$. We choose $\gamma_1=(a_2,c)$ with $c=\max\{a_2,a_3+1\}$ and $\gamma_2=(a_3+1,a_3+1)$. If $\omega_i=(c_1,c_2)\in I\setminus (N,a_2)$ and $\omega_j=(c_3,c_4)\in I$, then we distinguish two cases.
    \begin{enumerate}[1)]
      \item Suppose $\omega_i=(c_1,c_2)$ with $c_1\ge N+1$. Since $\omega_i\in I$, we have $c_2\ge N$. Hence, we have $[c_1,c_2,a_2,\cdots]\le \alpha$. Since $c\ge a_3+1$, then $\ [c_2,a_2,c,\cdots]\le \alpha$. Since $a_2\ge N+1$, then $[a_2,c,c_3,\cdots]\le \alpha$. Since $c\ge a_2\ge N+1$, then $[c,c_3,c_4,\cdots]\le \alpha$. Hence, we have $\omega_i\gamma_1\omega_j\in \mathcal{D}_6(\alpha)$.
      \item Suppose $\omega_i=(N,c_2)$ with $c_2\neq a_2$. Since $\omega_i\in I$, then $N\le c_2\le a_2-1$. By the same argument in the proof of $1)$, we also have $\omega_i\gamma_1\omega_j\in \mathcal{D}_6(\alpha)$.
    \end{enumerate}
    It can be checked that $(N,a_2)\gamma_2\omega_j$ is admissible for all $\omega_j\in I$. Hence, the matrix $A$ is finitely primitive.
    Now, we suppose $n\ge 4$. Fix any $\omega_0=(c_1,\cdots,c_{n-1})\in I$. Let $\gamma=(d_1,\cdots, d_{n-1})\in I$. We distinguish two cases.
    \begin{enumerate}[1)]
      \item Suppose $(c_k,\cdots,c_{n-1})\neq (a_1,\cdots,a_{n-k})$ for all $3\le k\le n-1$. If $(c_2,\cdots,c_{n-1})\neq (a_1,\cdots,a_{n-2})$, then take $d_1=\cdots=d_{n-1}=N+1$. Since $(c_k,\cdots,c_{n-1})\neq (a_1,\cdots,a_{n-k})$ for all $2\le k\le n-1$, then $[c_k,\cdots,N+1,\cdots,N+1,\cdots]\le\alpha$ for all $2\le k\le n-1$. Then we have $\omega_0\gamma \omega_1$ is admissible for all $\omega_1\in I$. If $(c_2,\cdots,c_{n-1})= (a_1,\cdots,a_{n-2})$, then we need only prove that for any $t_1\ge N$, we have $(c_2,\cdots,c_{n-1},d_1)$ and $(d_2,\cdots,t_1)$ are admissible. We distinguish two subcases.
      \begin{enumerate}[i)]
        \item If $n$ is odd, we claim that $a_{n-1}\ge N$. Suppose $a_{n-1}\le N-1$, then by Lemma \ref{zf}, we have $E(\alpha)=E([a_1,\cdots, a_{n-1}])$. Since $(c_1,\cdots ,c_{n-1})\in\mathcal{D}_{n}(\alpha)$, then $c_{n-1}\ge N>a_{n-1}$. Since $(c_2,\cdots, c_{n-2})=(a_1,\cdots,a_{n-2})$, then $(c_2,\cdots, c_{n-1})=(a_1,\cdots,a_{n-2},c_{n-1})$. Since $n-1$ is even, then $[c_2,\cdots,c_{n-1},\cdots]>[a_1,\cdots,a_{n-1}]$. This leads to a contradiction. Hence, we choose $d_1=\min \{a_{n-1},N+1\}$. Take $\gamma=(d_1,N+1,d_3,\cdots,d_{n-1})\in I$. Then, we complete the proof for this subcase.
        \item If $n$ is even, then we choose $d_1=a_{n-1}+1\ge N+1$. Take $\gamma=(d_1,N+1,\cdots,N+1)\in I$. Then, we complete the proof for this subcase.
      \end{enumerate}
      \item Suppose there exists some $k\ge 3$, such that $(c_k,\cdots,c_{n-1})= (a_1,\cdots,a_{n-k})$. Let $k_0$ be the smallest one. Consider $\alpha^{\prime}=(a_1,\cdots,a_{n-k_0+1})$. Since $n-k_0+1\le n-2$, by induction, there exists some $\gamma^{\prime}=(d_1,\cdots,d_{\ell})\in\mathcal{D}_{\ell}(\alpha^{\prime})=\mathcal{D}_{\ell}(\alpha)$ such that $(c_{k_0},\cdots,c_{n-1})\gamma \beta$ is admissible for all $\beta\in \mathcal{D}_{\ell}(\alpha)$.
      Then we can choose $\gamma=(\gamma^{\prime})^{m}$ where $m$ is the least common multiple of $\ell$ and $n-1$. Since $(c_{k_0},\cdots,c_{n-1})\gamma \beta$ is admissible for all $\beta\in \mathcal{D}_{\ell}(\alpha)$, we have $(d_k,\cdots,d_{\ell})\neq (a_1,\cdots,a_{\ell-k-1})$ for all $k\ge 1$. Then by the same arguments in the case $1)$, we complete the proof.
    \end{enumerate}
    Combining the two cases above, we have that $A$ is finitely primitive. By Lemmma \ref{wsdl}, we complete the proof for $\alpha=[N,a_2,\cdots,a_n]$ with $a_n\ge 2$ and $a_2\ge N+1$.
    \item Suppose $\alpha=[N,a_2,\cdots,a_n]$ with $a_n\ge 2$ and $a_2\le N$. Since $E\left(\frac{1}{N+1}\right)\subset E(\alpha)\subset E([N,N+1])$, we need only prove $\dim_H\left(E([N,N+1])\right)=\dim_H\left(E\left(\frac{1}{N+1}\right)\right)$. Denote $[\overline{N}]=[N,N,\cdots]=\frac{-N+\sqrt{N^2+4}}{2}$. We first prove 
    \begin{align*}
      E([N,N+1])\subset E\left(\frac{1}{N+1}\right)\cup \mathbb{Q}\cup \bigcup_{n=1}^{+\infty}\bigg\{[a_1,\cdots, a_n,\overline{N}]: a_i\ge N,\ \forall i\bigg\}\cup \{[\overline{N}]\}
    \end{align*}
    by contradiction. Suppose there exists some $x\in E([N,N+1])$ but $x\notin E\left(\frac{1}{N+1}\right)\cup \mathbb{Q}\cup \{[\overline{N}]\}$.
    Since $x=[x_1,\cdots]\in E([N,N+1])$, if there exists $i_0\ge 1$ such that $x_{i_0}\le N-1$, then $T^{i_0-1}(x)>[N,N+1]$ which leads a contradiction. Hence, we have $x_{i}\ge N$ for all $i\ge 1$. Since $x\notin E\left(\frac{1}{N+1}\right)$, then there exists some $i_1$ such that $x_{i_1}=N$. Since $x\notin \mathbb{Q}\cup \bigcup_{n=1}^{+\infty}\bigg\{[a_1,\cdots, a_n,\overline{N}]: a_i\ge N,\ \forall i\bigg\}$, then there exists some $i_2$ such that $x_{i_2}=N$ and $x_{i_2+1}\ge N+1$. Therefore, we have $T^{i_2-1}(x)>[N,N+1]$ which leads a contradiction. Hence, we have 
    $\dim_H\left(E([N,N+1])\right)\le \dim_H\left(E\left(\frac{1}{N+1}\right)\right)$. Since the function $\alpha\mapsto E(\alpha)$ is increasing, we have $\dim_H(E(\alpha))=\dim_H(E([N,N+1]))$.
    Since $\alpha\mapsto P(t,\alpha)$ is increasing, then for any $t\ge 0$,
    \[P\left(t,\frac{1}{N+1}\right)\le P(t,\alpha)\le P(t,[N,N+1]).\] 
    This implies \eqref{wsgs2}.
  \end{enumerate}
  Combining the two cases above, we complete the proof for all $\alpha\in \mathbb{Q}\setminus \{0\}$. Since $E(0)=\{0\}$, then $\dim_H(E(0))=0$. This completes the proof.
\end{proof}
\section{The continuity of the function $\alpha\mapsto \dim_H(E(\alpha))$}
In this section, we denote the function $\alpha\mapsto \dim_H(E(\alpha))$ by $d$, i.e., 
\[d:[0,1]\rightarrow [0,1],\ \alpha\mapsto \dim_H(E(\alpha)).\]
Let $\phi_a(x)=1/(a+x)$, and for a word $w=(a_1,\ldots,a_n)$ write
$\phi_w=\phi_{a_1}\circ\cdots\circ\phi_{a_n}$ and
$I_w=\phi_w((0,1))$. With $F_1=F_2=1$ the Fibonacci numbers, it can be checked that for any finite words $u$ and $v$,
\begin{equation}\label{basic:continuants}
 q_uq_v\le q_{uv}\le2q_uq_v,
 \qquad q_v\ge F_{n+1},\qquad q_v\ge\prod_{i=1}^n a_i.
\end{equation}
The inverse branch satisfies
\begin{equation}\label{basic:derivatives}
 \frac1{4q_w^2}\le |\phi_w'(x)|\le\frac1{q_w^2}.
\end{equation}
In particular the distortion is at most $4$, the derivative magnitude is monotone, and
\begin{equation}\label{basic:sum}
 \sum_{|w|=n}q_w^{-2t}\le
 \left(\sum_{a\ge1}a^{-2t}\right)^n<\infty\qquad(t>\tfrac12).
\end{equation}
For each fixed $n$, the sum over words with at least one digit greater than
$M$ therefore tends to zero as $M\to\infty$.\par
We use the following elementary finite-IFS principle. For finitely many contracting interval maps with disjoint image interiors, if $\sum_i(\inf|\psi_i'|)^t>1$, their limit set has dimension greater than $t$. Indeed the strict inequality persists at an exponent $t'>t$; the Bernoulli measure with weights proportional to $(\inf|\psi_i'|)^{t'}$ gives the usual mass-distribution lower bound. Below we apply this principle only to finite collections of continued-fraction inverse branches, for which \eqref{basic:derivatives} gives bounded distortion and the open set condition holds.
\begin{lemma}\label{basic:half}
For each $\alpha>0$, $d(\alpha)>1/2$.
\end{lemma}
\begin{proof}
  For any $\alpha>0$, there exists some integer $M\ge 1$ such that $\alpha>\frac{1}{M}$. Hence, we have 
  \[d(\alpha)\ge d\left(\frac{1}{M}\right).\]
  On the other hand,
  \[E\left(\frac{1}{M}\right)=\{x=[x_1,\cdots]\in [0,1]: x_i\ge M,\ \forall i\ge 1\}.\]
  By \cite{G41}, we have 
  \[d\left(\frac{1}{M}\right)> \lim_{M\to +\infty}d\left(\{x=[x_1,\cdots]\in [0,1]: x_i\ge M,\ \forall i\ge 1\}\right)=\frac{1}{2}.\]
 This completes the proof.
\end{proof}
\subsection{A reference measure and the operator spaces}
For $t>1/2$, define the full and open transfer operators from $C([0,1])$ to itself by
\begin{equation*}%\label{op:definition}
 \begin{split}
 (\Lop_t f)(x)&=\sum_{a\ge1}|\phi^{\prime}_a(x)|^tf(\phi_a(x)),\\
 \Lop_{t,\alpha}f&=\Lop_t(\ind_{[0,\alpha]}f).
 \end{split}
\end{equation*}
For spectral perturbations it suffices to fix $t\in(1/2,1)$; exponents outside this range will not be needed in the continuity argument.
\begin{lemma}\label{op:measure}
For $1/2<t<1$ there is an atomless, full-support probability measure $\nu_t$
and a number $\lambda_t>1$ such that
$\Lop_t^*\nu_t=\lambda_t\nu_t$.
\end{lemma}
\begin{proof}
  Since 
  \[\Lop_t\ind=\sum_{a\ge 1}(x+a)^{-2t}\ge 2^{-2t},\]
  then we can define a map 
  \[\mu\longmapsto\frac{\Lop_t^*\mu}{\int\Lop_t\ind\,d\mu}\]
  on the compact convex set of probability measures. This map is weak-star continuous and has a fixed point by the Schauder fixed point theorem. Let $\nu_t$ be the fix point, i.e.,
  \[\nu_t=\frac{\Lop_t^*\nu_t}{\int\Lop_t\ind\,d\nu_t}.\]
  Denote $\int\Lop_t\ind\,d\nu_t$ by $\lambda_t$. Then, we have $\Lop_t^*\nu_t=\lambda_t\nu_t$.\par
  Let $h(x)=1/(1+x)$. It can be checked that $\Lop_1h=h$.
  By \eqref{basic:derivatives}, every two-digit inverse branch has derivative magnitude at most $1/4$, thus
  \begin{align*}
    \Lop_t^2h(x) & = \Lop_t\left(\sum_{a\ge1}|\phi^{\prime}_a(x)|^th(\phi_a(x))\right)\\
    & = \sum_{a,b\ge1}|(\phi_b\circ \phi_a)^{\prime}(x)|^th(\phi_b\circ \phi_a(x))\\
    & \ge 4^{1-t}\sum_{a,b\ge1}|(\phi_b\circ \phi_a)^{\prime}(x)|h(\phi_b\circ \phi_a(x))\\
    & = 4^{1-t} \Lop_1^2h(x).
  \end{align*}
  Since $\Lop_t^*\nu_t=\lambda_t\nu_t$, we have 
  \[\int \Lop_t^2h(x)\,d\nu_t=\lambda_t^2\int h(x)\,d\nu_t.\]
  On the other hand 
  \[\int \Lop_t^2h(x)\,d\nu_t\ge 4^{1-t}\int \Lop_1^2h(x)\,d\nu_t=4^{1-t}\int h(x)\,d\nu_t.\]
  Hence, we have $\lambda_t\ge 2^{1-t}>1$. \par
  Since the eigenmeasure relation can be extended to bounded Borel functions. We first choose $f=\ind_{\{0\}}$. Since no inverse branch takes the value $0$, then
  \[\Lop_t\ind_{\{0\}}(x)=0.\]
  Hence, we have $\nu_t(\{0\})=0$.
  For other $y>0$, we have 
  \begin{align}\label{yuanzi}
    \lambda_t\nu_t(\{y\}) = \int \Lop_t\ind_{\{y\}}\,d\nu_t
     = y^{2t}\sum_{\substack{a\ge 1\\ y^{-1}-a\in [0,1]}}\nu_t\left(\left\{\frac{1}{y}-a\right\}\right).
  \end{align}
  If $y=1$, by \eqref{yuanzi}, we have $\nu_t(\{1\})=\nu_t(\{0\})=0$. If $y=\frac{1}{m}$ for some $m\ge 2$, by \eqref{yuanzi}, we have 
  $\nu_t(\{m^{-1}\})=m^{-2t}(\nu_t(\{0\})+\nu_t(\{1\}))=0$.
  For other $y$, by \eqref{yuanzi}, we have
  \[\lambda_t\nu_t(\{y\})=y^{2t}\nu_t(\{Ty\}).\]
  Hence,
\[
 \lambda_t^n\nu_t(\{y\})=\prod_{j=0}^{n}(T^jy)^{2t}\nu_t(\{T^ny\})=|(T^n)'(y)|^{-t}\nu_t(\{T^ny\})\le 1.
\]
Since $\lambda_t>1$, we have $\nu_t(\{y\})=0$. Therefore, the measure $\nu_t$ is atomless.\par
Finally, for any nonempty open interval $U$, choose $w$ with $I_w\subset U$. Then, $\Lop_t^{|w|}\ind_U$ is strictly positive on $(0,1)$ by \eqref{basic:derivatives}; thus $\nu_t(U)>0$.
\end{proof}
Use the usual space of bounded-variation functions modulo Lebesgue almost-everywhere equality, with essential variation and canonical one-sided representatives, and equip it with the norms
\[
 |f|_t=\int|f|\,d\nu_t,
 \qquad \|f\|_t=\Var(f)+|f|_t.
\]
These strong norms are equivalent to $\Var(f)+\|f\|_\infty$:
$\|f\|_\infty\le\Var(f)+|f|_t$. Essential variation and canonical one-sided representatives remove point-value ambiguities. Since $\nu_t$ is atomless and has full support, the equivalence class and the operator are well-defined. Helly selection and dominated convergence show that the
strong unit ball is compact in the weak norm $|\cdot|_t$.
\subsection{Uniform spectral stability with countably many branches}
\begin{lemma}\label{op:LY}
For each fixed $t\in(1/2,1)$ there are $n_0\ge1$, $0<\eta<1$, and $C<\infty$
such that, uniformly in $\alpha\in[0,1]$,
\begin{align}
 |\Lop_{t,\alpha}^{\,k}f|_t&\le\lambda_t^k|f|_t\qquad(k\ge0),
 \label{op:weak}\\
 |(\Lop_{t,\alpha}-\Lop_{t,\beta})f|_t
 &\le\lambda_t\nu_t((\min\{\alpha,\beta\},\max\{\alpha,\beta\}])\|f\|_t.\label{op:perturb}
\end{align}
Moreover,
\begin{equation}\label{op:LYestimate}
 \|\Lop_{t,\alpha}^{\,n_0}f\|_t\le\eta\|f\|_t+C|f|_t.
\end{equation}
\end{lemma}
\begin{proof}
  First, we prove the weak norm estimate \eqref{op:weak}. Since the weights in the sum are all positive, we have
  \begin{equation*}
    |\mathcal{L}_{t,\alpha} f| \leq \mathcal{L}_{t,\alpha} |f| = \mathcal{L}_{t}\!\left(\mathbf{1}_{[0,\alpha]} |f|\right) \leq \mathcal{L}_{t} |f|.
  \end{equation*}
  Integrating with respect to $\nu_{t}$ and using the eigenmeasure relation, we have 
  \begin{equation*}
    |\mathcal{L}_{t,\alpha} f|_{t} \leq \int \mathcal{L}_{t} |f| \, d\nu_{t} = \lambda_{t} \int |f| \, d\nu_{t} = \lambda_{t} |f|_{t}.
  \end{equation*}
  Iterating this one-step estimate yields the conclusion for arbitrary $k$.\par
  Now, we prove \eqref{op:perturb}. Assume that $\alpha \leq \beta$. By definition,
  \begin{equation*}
    (\mathcal{L}_{t,\alpha} - \mathcal{L}_{t,\beta}) f = -\mathcal{L}_{t}\!\left(\mathbf{1}_{(\alpha,\beta]} f\right).
  \end{equation*}
  This shows that the difference between the two operators comes only from the newly added interval $(\alpha,\beta]$.
  Therefore,
  \begin{align*}
    |(\mathcal{L}_{t,\alpha} - \mathcal{L}_{t,\beta}) f|_{t} & \leq \int \mathcal{L}_{t}\!\left(\mathbf{1}_{(\alpha,\beta]} |f|\right) d\nu_{t} \\
    & = \lambda_{t} \int_{(\alpha,\beta]} |f| \, d\nu_{t} \\
    & \leq \lambda_{t} \nu_{t}\bigl((\alpha,\beta]\bigr) \|f\|_{\infty} \\
    & \leq \lambda_{t} \nu_{t}\bigl((\alpha,\beta]\bigr) \|f\|_{t}.
  \end{align*}
  Finally, we prove the strong norm estimate \eqref{op:LYestimate}. The first step is to decompose the iterated operator according to the cylinder sets. Fix a length $n$. For any word $w=[a_1,\cdots,a_n]$, let
  \[K_{w,\alpha}=I_{w}\cap\bigcap_{j=0}^{n-1}\left\{y:T^jy\le \alpha\right\}.\]
  Since each $T^j$ is monotone on a full Gauss cylinder $I_w$, then $K_{w,\alpha}$ is an interval (possibly empty or degenerate). Let 
  \[J_{w,\alpha}=T^n(K_{w,\alpha}).\]
  By the definition of $\mathcal{L}_{t,\alpha}$, we have
  \[\mathcal{L}^n_{t,\alpha}f=\sum_{|w|=n}\ind_{J_{w,\alpha}}|\phi^{\prime}_{w}|^t(f\circ \phi_w).\]
  Denote $\ind_{J_{w,\alpha}}|\phi^{\prime}_{w}|^t(f\circ \phi_w)$ by $B_{w,\alpha}f$ for convenience. Since $|\phi^{\prime}_{w}|^t$ is monotonic, we have
  \[\Var(|\phi^{\prime}_{w}|^t)\le \sup |\phi^{\prime}_{w}|^t.\]
  By \eqref{basic:derivatives}, we have 
  \[\Var(|\phi^{\prime}_{w}|^t)\le q_{w}^{-2t}.\]
  Since $\phi_w$ is monotonic, then
  \begin{equation*}
    \Var_{J_{w,\alpha}}\!\left(|\phi^{\prime}_{w}|^t(f \circ \phi_{w})\right) \leq q_w^{-2t} \Var_{I_{w}}(f) + q_w^{-2t} \sup_{I_{w}} |f|.
  \end{equation*}
  Multiplying by $\mathbf{1}_{J_{w,\alpha}}$ amounts to truncating this function to zero outside $J_{w,\alpha}$. This creates at most two additional jumps, located at the two endpoints of $J_{w,\alpha}$, and the size of each jump is at most
  \begin{equation*}
    q_w^{-2t} \sup_{I_{w}} |f|.
  \end{equation*}
  Therefore,
  \begin{align}\label{varBf1}
    \Var\!\left(B_{w,\alpha} f\right) \leq q_w^{-2t}\left(\Var_{I_{w}}(f)+3\sup_{I_{w}} |f|\right).
  \end{align}
  On the other hand, for any $y,z\in I_{w}$, by the definition of $\Var$, we have
  \[|f(y)|\le |f(z)|+\Var_{I_{w}}(f).\]
  Therefore,
  \[\int_{I_w} |f(y)|\,d\nu_t(z)\le \int_{I_w} \big(|f(z)|+\Var_{I_{w}}(f)\big)\,d\nu_t(z).\]
  By Lemma \ref{op:measure}, we have $\nu_t(I_w)>0$. Hence,
  \[\sup_{I_{w}} |f|\le \Var_{I_{w}}(f)+\frac{1}{\nu_t(I_w)}\int_{I_w}|f|\,d\nu_t.\]
  By \eqref{varBf1}, we have 
  \begin{align}\label{varBf2}
    \Var\!\left(B_{w,\alpha} f\right) \leq q_w^{-2t}\left(4\Var_{I_{w}}(f)+\frac{3}{\nu_t(I_w)}\int_{I_w}|f|\,d\nu_t\right).
  \end{align}
  Fix a positive integer $M\ge 1$. We first sum only over the finite family
  \begin{equation*}
    w \in \{1, \ldots, M\}^{n}.
  \end{equation*}
  The interiors of these cylinder sets of the same order are pairwise disjoint, and therefore
  \begin{equation*}
    \sum_{w\in \{1, \ldots, M\}^{n}} \Var_{I_{w}}(f) \leq \Var(f),
    \qquad
    \sum_{w\in \{1, \ldots, M\}^{n}} \int_{I_{w}} |f| \, d\nu_{t} \leq |f|_{t}.
  \end{equation*}
  Since only finitely many cylinder sets are involved, we can take the finite constant
  \begin{equation*}
    C_{n,M,t} := 3 \max_{w \in \{1, \ldots, M\}^{n}} \frac{q_{w}^{-2t}}{\nu_{t}(I_{w})}.
  \end{equation*}
  By \eqref{basic:continuants} and \eqref{varBf2}, we have 
  \begin{equation}\label{varBf3}
    \sum_{w \in \{1, \ldots, M\}^{n}} \Var\!\left(B_{w,\alpha} f\right) \leq 4 F_{n+1}^{-2t} Var(f) + C_{n,M,t} |f|_{t}.
  \end{equation}
  On the other hand, by \eqref{varBf1}, 
  \begin{align*}
    \Var\!\left(B_{w,\alpha} f\right) & \leq q_w^{-2t}\left(\Var(f) + 3\|f\|_{\infty}\right)\\
    & \leq 4 q_w^{-2t} \|f\|_{t}.
  \end{align*}
  Hence, 
  \begin{equation}\label{varBf4}
    \sum_{\substack{|w| = n \\ \max_{i} a_{i} > M}} \Var\!\left(B_{w,\alpha} f\right) \leq 4 \left( \sum_{\substack{|w| = n \\ \max_{i} a_{i} > M}} q_{w}^{-2t} \right) \|f\|_{t}.
  \end{equation}
  These summable weights also guarantee the convergence of the series of branches, which justifies the use of
  \begin{equation*}
    \Var\!\left( \sum_{w} B_{w,\alpha} f \right) \leq \sum_{w} \Var\!\left(B_{w,\alpha} f\right)
  \end{equation*}
  for the infinite sum. By \eqref{varBf3} and \eqref{varBf4}, we have
  \begin{align*}
    \Var\!\left(\mathcal{L}_{t,\alpha}^{n} f\right) & \leq 4 F_{n+1}^{-2t} \Var(f) + C_{n,M,t} |f|_{t} \\
    & \quad + 4 \left( \sum_{\substack{|w| = n \\ \max_{i} a_{i} > M}} q_{w}^{-2t} \right) \|f\|_{t}.
  \end{align*}
  Since $F_{n+1} \to \infty$, we first fix $n$ such that
  \begin{equation*}
    4 F_{n+1}^{-2t} < \frac{1}{4}.
  \end{equation*}
  Then, keeping this $n$ fixed, we choose $M$ large enough such that
  \begin{equation*}
    4 \sum_{\substack{|w| = n \\ \max_{i} a_{i} > M}} q_{w}^{-2t} < \frac{1}{4}.
  \end{equation*}
  Since $\Var(f) \leq \|f\|_{t}$, we obtain
  \begin{equation*}
    \Var\!\left(\mathcal{L}_{t,\alpha}^{n} f\right) \leq \frac{1}{2} \|f\|_{t} + C_{n,M,t} |f|_{t}.
  \end{equation*}
  By \eqref{op:weak}, we have 
  \begin{equation*}
    \|\mathcal{L}_{t,\alpha}^{n} f\|_t=\Var\!\left(\mathcal{L}_{t,\alpha}^{n} f\right)+|\mathcal{L}_{t,\alpha}^{n} f|_t\le \frac{1}{2} \|f\|_{t} + \left(C_{n,M,t}+\lambda_t^n\right) |f|_{t}.
  \end{equation*}
\end{proof}
For a bounded operator $A$ on this space, write $\sigma(A)$ for its
spectrum and
\[
 \rho(A)=\sup\{|z|:z\in\sigma(A)\}
\]
for its spectral radius.
\begin{proposition}[Uniform spectral stability of the open operators]\label{op:spectralstability}
  Fix $t\in(1/2,1)$. There is a constant $\kappa_t\in(0,1)$,
  independent of $\alpha\in[0,1]$, such that
  \begin{equation}\label{op:essentialradius}
   r_{\mathrm{ess}}(\Lop_{t,\alpha})\le\kappa_t
   \qquad(\alpha\in[0,1]),
  \end{equation}
  where $r_{\mathrm{ess}}$ denotes the essential spectral radius. Moreover, the following assertions hold.
  \begin{enumerate}
   \renewcommand{\labelenumi}{(\roman{enumi})}
   \renewcommand{\theenumi}{\roman{enumi}}
   \item\label{op:spectralupper}
  If $\alpha_j\to\alpha$ in $[0,1]$ and
  $z_j\in\sigma(\Lop_{t,\alpha_j})$ satisfy $z_j\to z$ with
  $|z|>\kappa_t$, then $z\in\sigma(\Lop_{t,\alpha})$.
  \item\label{op:spectralpersistence}
  Every $z\in\sigma(\Lop_{t,\alpha})$ with $|z|>\kappa_t$ is an isolated eigenvalue of finite algebraic multiplicity.
  Let $D(z,\delta)=\{w\in\mathbb C:|w-z|<\delta\}$, where
  $\delta>0$ is small enough that
  \[\overline{D(z,\delta)}\subset\{w:|w|>\kappa_t\},
   \qquad\overline{D(z,\delta)}\cap\sigma(\Lop_{t,\alpha})=\{z\}.\]
  For all $\beta\in[0,1]$ sufficiently close to $\alpha$, the boundary $\partial D(z,\delta)$ is disjoint from $\sigma(\Lop_{t,\beta})$, and the eigenvalues of $\Lop_{t,\beta}$ in $D(z,\delta)$ have total algebraic multiplicity equal to that of $z$. In particular, for every sequence $\alpha_j\to\alpha$, there are, for all sufficiently large $j$, spectral values $z_j\in\sigma(\Lop_{t,\alpha_j})$ such that $z_j\to z$.
  \end{enumerate}
\end{proposition}
\begin{proof}
  By Lemma~\ref{op:LY}, the weak bound \eqref{op:weak} and the uniform Lasota--Yorke estimate \eqref{op:LYestimate} hold. Iterating \eqref{op:LYestimate} and treating the finitely many remaining powers therefore gives
  \begin{equation}\label{op:twonorm}
   \|\Lop_{t,\alpha}^{\,k}f\|_t \le C_1\kappa_t^k\|f\|_t+C_2A_t^k|f|_t \qquad(k\ge0).
  \end{equation}
  The compactness of the strong unit ball in the weak norm, together with \eqref{op:twonorm}, yields \eqref{op:essentialradius}. Consequently, all spectral values outside the closed disk of radius $\kappa_t$ are isolated eigenvalues of finite algebraic multiplicity. Define the strong-to-weak operator norm by
  \[\|A\|_{t\to w}=\sup_{\|f\|_t\le1}|Af|_t.\]
  The perturbation estimate \eqref{op:perturb} gives
  \[\|\Lop_{t,\beta}-\Lop_{t,\alpha}\|_{t\to w}
   \le\lambda_t\nu_t\bigl((\min\{\alpha,\beta\},
                              \max\{\alpha,\beta\}]\bigr)
   \longrightarrow0\qquad(\beta\to\alpha).\]
  Here atomlessness of $\nu_t$ ensures the convergence; its distribution function is uniformly continuous on $[0,1]$, so this estimate also provides a uniform modulus of perturbation. Thus the hypotheses of the Keller--Liverani theorem are satisfied.\par
  For \eqref{op:spectralupper}, suppose that $z\notin\sigma(\Lop_{t,\alpha})$ and choose $r$ with $\kappa_t<r<|z|$.
  Theorem~1 of \cite{KL99} makes a sufficiently small closed
  neighborhood of $z$ disjoint from $\sigma(\Lop_{t,\beta})$ for all $\beta$ close to $\alpha$. This contradicts $z_j\to z$.\par
  For \eqref{op:spectralpersistence}, choose $\kappa_t<r<\min_{w\in\overline{D(z,\delta)}}|w|$. Theorem~1 and Corollary~1 of \cite{KL99} give stability of the spectral projection associated with $\partial D(z,\delta)$: its finite rank is unchanged for $\beta$ sufficiently close to $\alpha$. This rank equals the total algebraic multiplicity of the eigenvalues inside the disk, proving the stated equality. Applying the same conclusion to arbitrarily small disks centered at $z$ gives the final sequential assertion.
\end{proof}
\subsection{Rational parameters and the dimension threshold}
For completeness, the next lemma supplies the spectral identification needed below, independently of a global finite-primitivity claim. The BV spectral radius does not depend on the equivalent strong norm used. 
\begin{lemma}\label{rat:crossing}
Let $\alpha\in\mathbb Q\cap(0,1]$. Then $d(\alpha)>1/2$ and
\[
 \rho(\mathcal{L}_{t,\alpha})
 \begin{cases}
 >1, & 1/2<t<d(\alpha),\\
 =1, & t=d(\alpha),\\
 <1, & t>d(\alpha).
 \end{cases}
\]
Moreover, with $\varphi=(1+\sqrt5)/2$, we have
\begin{equation}\label{rat:radius-slope}
 \rho(\mathcal{L}_{t,\alpha})\le\varphi^{-2(t-s)}\rho(\mathcal{L}_{s,\alpha})
 \qquad(t>s>1/2).
\end{equation}
\end{lemma}
\begin{proof}
  We use the Markov coding, the irreducibility verification, and the rational-parameter Bowen formula established in Section~4 of the manuscript. We do not repeat these arguments. Let 
  \begin{equation}\label{rat:pressure}
   Q_n(t,\alpha)=\sum_{w\in\mathcal D_n(\alpha)}q_w^{-2t}.
  \end{equation}
  Now, we identify the pressure $P(t,\alpha)$ with the spectral radius on $BV$.\par
  Since $\alpha\in \mathbb{Q}$, there exists some $m\ge 1$ such that $T^m(\alpha)=0$ but $T^{m-1}(\alpha)\neq 0$. Define
  \begin{align*}
    B_{\alpha}=\left(\{0,1\}\cup \{T^j(\alpha): 1\le j\le m\}\right)\cap [0,\alpha].
  \end{align*}
  Since $B_{\alpha}$ consists of finitely many points, we can arrange them in increasing order
  \begin{equation*}
    0 = b_{0} < b_{1} < \cdots < b_{r} = \alpha.
  \end{equation*}
  They divide $[0,\alpha]$ into finitely many open intervals
  \begin{equation*}
    H_{i} = (b_{i-1}, b_{i}).
  \end{equation*}
  We then keep only those intervals that can support orbits surviving forever, that is, those satisfying
  \begin{equation*}
    H_{i} \cap \bigl( E(\alpha) \setminus \mathbb{Q} \bigr) \neq \varnothing,
  \end{equation*}
  and renumber them as
  \begin{equation*}
    J_{1}, \ldots, J_{K}.
  \end{equation*}
  Each $J_i$ only needs to contain at least one point that survives forever; the other points in the interval may still escape later. For fixed $\alpha$, these intervals and their number $K$ are fixed, and do not change with the number $n$ of iterations. Set 
  \[U_{\alpha}=\bigcup_{i=1}^KJ_i,\ \quad \mathcal{B}=\bigoplus_{i=1}^K BV(J_i).\]
  Define an operator $\mathcal{R}_{t,\alpha}$ on $\mathcal{B}$ by 
  \[\mathcal{R}_{t,\alpha}f:=\mathcal{L}_{t,\alpha}\ind_{U_{\alpha}}f \bigr|_{U_\alpha},\]
  i.e.,
  \begin{equation*}
    (\mathcal{R}_{t,\alpha} f)(x) = \sum_{\substack{k \geq 1 \\ 1/(a+x) \in U_{\alpha}}}|\phi^{\prime}_a(x)|^t f\!\left(\phi_a(x)\right),
    \qquad x \in U_{\alpha}.
  \end{equation*}
  Now, we claim that 
  \begin{equation}\label{suanzifenjie}
    \mathcal{L}_{t,\alpha} = \begin{pmatrix} \mathcal{R}_{t,\alpha} & 0 \\ \mathcal{C}_{t,\alpha} & \mathcal{N}_{t,\alpha} \end{pmatrix}, \quad \mathcal{N}_{t,\alpha}^m=0\  \text{for some}\ m.
  \end{equation}
  Let $V_{\alpha}$ denote the union of the remaining intervals,i.e.,
  \[V_{\alpha}=[0,1]\setminus \left(B_{\alpha}\cup U_{\alpha}\right).\]
  In order to prove \eqref{suanzifenjie}, we 
  write any input function as
  \begin{equation*}
    f = \begin{pmatrix} f_{U_{\alpha}} \\ f_{V_{\alpha}} \end{pmatrix}.
  \end{equation*}
  The output function can also be split into two parts according to $U_{\alpha}$ and $V_{\alpha}$. Hence the original operator can first be written in the general block form
  \begin{equation*}
    \mathcal{L}_{t,\alpha} = \begin{pmatrix} \mathcal{R}_{t,\alpha} & \mathcal{B}_{t,\alpha} \\ \mathcal{C}_{t,\alpha} & \mathcal{N}_{t,\alpha} \end{pmatrix}.
  \end{equation*}
 The columns indicate where the input function lives, and the rows indicate where the output lives. The four blocks are as follows:
 \begin{center}
  \begin{tabular}{c|c|c}
   Block & Input part & Output part \\
   \hline
   $\mathcal{R}_{t,\alpha}$ & $U_{\alpha}$ & $U_{\alpha}$ \\
   $\mathcal{B}_{t,\alpha}$ & $V_{\alpha}$ & $U_{\alpha}$ \\
   $\mathcal{C}_{t,\alpha}$ & $U_{\alpha}$ & $V_{\alpha}$ \\
   $\mathcal{N}_{t,\alpha}$ & $V_{\alpha}$ & $V_{\alpha}$
   \end{tabular}
  \end{center}
  We first prove $\mathcal{B}_{t,\alpha}=0$. We decompose the inputs in $V_{\alpha}$ to two types. The first type lies in $(\alpha, 1]$. By definition,
  \begin{equation*}
    \mathcal{L}_{t,\alpha} f = \mathcal{L}_{t}\!\left( \mathbf{1}_{[0,\alpha]} f \right),
  \end{equation*}
  so the operator simply ignores this part of the input.
  The second type lies in $[0,\alpha]$, but belongs to an interval state that cannot be continued indefinitely. Such a state cannot have a Markov edge leading to some $J_{i}$: otherwise, concatenating this edge with an infinitely surviving path starting from $J_{i}$ would yield an infinitely surviving path starting from this state, contradicting the definition of $V_\alpha$. Therefore, inputs from $V_{\alpha}$ contribute nothing to the output on $U_{\alpha}$, that is,
  \begin{equation*}
    \mathcal{B}_{t,\alpha} = 0.
  \end{equation*}
  By similar arguments above, we have that the corresponding state diagram of $\mathcal{N}_{t,\alpha}$ has only finitely many vertices and contains no directed cycles. This implies that $\mathcal{N}_{t,\alpha}$ is nilpotent. Hence, the equation \eqref{suanzifenjie} holds.\par
  By \eqref{suanzifenjie}, we have 
  \[\rho(\mathcal{L}_{t,\alpha})=\rho(\mathcal{R}_{t,\alpha}).\]
  For a length $n$ path $p$ from $J_u$ to $J_v$, define its inverse branch by $\phi_p:J_v\rightarrow J_u$. Let
  \begin{equation*}
    b_{t}(p) = \sup_{J_{v}} |\phi_{p}'|^{t},
    \qquad
    S_{n}(t) = \sum_{|p| = n} b_{t}(p).
  \end{equation*}
  If the path $p$ corresponds to the partial quotient word $w$, by \eqref{basic:derivatives}, we have 
  \begin{equation*}
    4^{-t} q_{w}^{-2t} \leq b_{t}(p) \leq q_{w}^{-2t}.
  \end{equation*}
  Note that if the word and the terminal interval are fixed, then all the intermediate intervals are determined as well. Therefore, each word is realized by at least one and at most $K$ paths, and the multiplicity does not grow with $n$. Hence,
  \begin{equation}\label{A}
    4^{-t} Q_{n}(t,\alpha) \leq S_{n}(t) \leq K Q_{n}(t,\alpha).
  \end{equation}
  Divide the paths into $K$ groups according to their terminal intervals. The total weight is $S_{n}$, so at least one group has weight no less than $S_{n}/K$. Within this terminal interval, bounded distortion guarantees that the actual weight of each path is no less than $4^{-t}$ times its maximal weight. Therefore
  \begin{equation*}
    \mathcal{R}_{t,\alpha}^{n} \mathbf{1}(x) \geq \frac{S_{n}(t)}{K 4^{t}}.
  \end{equation*}
  By the definition of the operator norm, we have
  \begin{equation}\label{xianjie}
    \bigl\| R_{t,\alpha}^{n} \bigr\| \geq \frac{S_{n}(t)}{K 4^{t}}.
  \end{equation}
  On the other hand, let
  \begin{equation*}
    M = \max_{v} \|f_{v}\|_{\infty},
    \qquad
    V = \sum_{v} \Var_{J_{v}}(f_{v}).
  \end{equation*}
  Then $\|f\|_{\mathcal{B}} = M + V$. The sum of the maximal weights over all branches is $S_{n}$, and therefore the supremum of the output satisfies
  \begin{equation*}
    \|R_{t,\alpha}^{n} f\|_{\infty} \leq S_{n} M.
  \end{equation*}
  For the contribution of each path, we use the product rule for variation together with the monotonicity of the weights
  \begin{equation*}
    \Var\!\left( |\phi_{p}'|^{t} (f \circ \phi_{p}) \right) \leq b_{t}(p) \Var_{J_{u}}(f) + b_{t}(p) M.
  \end{equation*}
  Summing over the paths gives
  \begin{equation*}
    \sum_{v} \Var_{J_{v}}\!\left( R_{t,\alpha}^{n} f \right) \leq S_{n} V + S_{n} M.
  \end{equation*}
  Hence,
  \begin{equation*}
    \|R_{t,\alpha}^{n} f\|_{\mathcal{B}} \leq S_{n} V + 2 S_{n} M \leq 2 S_{n} \|f\|_{\mathcal{B}}.
  \end{equation*}
  Therefore,
  \begin{equation}\label{shangjie}
    \bigl\| R_{t,\alpha}^{n} \bigr\| \leq 2 S_{n}(t).
  \end{equation}
  By \eqref{A}, \eqref{xianjie} and \eqref{shangjie}, we obtain
  \[\frac{Q_n(t,\alpha)}{K4^s}\le \bigl\| R_{t,\alpha}^{n} \bigr\|\le 2KQ_n(t,\alpha).\]
  Taking $n$th roots and applying the spectral radius formula proves
  \begin{equation}\label{rat:spectralpressure}
    \rho(\mathcal{L}_{t,\alpha})=\rho(\mathcal{R}_{t,\alpha})=e^{P(t,\alpha)}.
  \end{equation}
  Finally, since $q_w\ge F_{n+1}$ for every length-$n$ digit word, then
  \[ Q_n(t,\alpha)\le F_{n+1}^{-2(t-s)}Q_n(s,\alpha)\qquad(t>s>1/2).\]
  Since $n^{-1}\log F_{n+1}\to\log\varphi$, this implies
  \begin{equation}\label{rat:strictslope}
   P(t,\alpha)\le P(s,\alpha)-2(t-s)\log\varphi.
  \end{equation}
  Continuity of $P(\cdot,\alpha)$ at $d(\alpha)$ and \eqref{wsgs2} give $P(d(\alpha),\alpha)=0$. Equation~\eqref{rat:strictslope} now gives positive pressure below $d(\alpha)$ and negative pressure above it. Combining this with \eqref{rat:spectralpressure} proves the three spectral signs and
  \eqref{rat:radius-slope}.
\end{proof}
\subsection{The proof of Theorems \ref{wsgs} and \ref{lx}}
We first prove Theorem \ref{lx}, then by Theorem \ref{lx} and Proposition \ref{wsgs1}, we complete the prove of Theorem \ref{wsgs}.
\begin{proof}[Proof of Theorem \ref{lx}]
  Fix $\alpha_0\in(0,1]$. Let
  \[d_- =\sup_{0<\beta<\alpha_0}d(\beta),\qquad
    d_+ =\begin{cases}
    \inf_{\alpha_0<\beta\le1}d(\beta),&\alpha_0<1,\\
    d(1)=1,&\alpha_0=1.
 \end{cases}\]
  Since the function $d$ is nondecreasing, then $d_-\le d(\alpha_0)\le d_+$. By Lemma~\ref{basic:half} gives $d_->1/2$. Suppose $d_-<d_+$. Then, we can choose
  \[d_-<t<d_+, \qquad 0<\varepsilon<\min\{t-d_-,d_+-t\}.\]
  In particular $1/2<t<1$. Take positive rationals $\beta_j\uparrow\alpha_0$. If $\alpha_0<1$, take rationals
  $\gamma_j\downarrow\alpha_0$ from above, if $\alpha_0=1$, take $\gamma_j\equiv 1$. Then,
  \[d(\beta_j)\le t-\varepsilon, \qquad d(\gamma_j)\ge t+\varepsilon.\]
  By \eqref{rat:radius-slope}, we have
  \begin{align}\label{C}
    \rho(\mathcal{L}_{t,\beta_j})\le \varphi^{-2(t-d(\beta_j))}\rho(\mathcal{L}_{d(\beta_j),\beta_j})\le \varphi^{-2\varepsilon}<1.
  \end{align}
  The same arguments gives 
  \begin{align}\label{D}
    \rho(\mathcal{L}_{t,\gamma_j})\ge \varphi^{2\varepsilon}>1.
  \end{align}
  For each $j$ choose a spectral value $z_j$ of $\Lop_{t,\gamma_j}$ whose modulus is its spectral radius. The operators have uniformly bounded strong norms, so a subsequence converges to $z$ with $|z|\ge\varphi^{2\varepsilon}>1$. Proposition~\ref{op:spectralstability}(\ref{op:spectralupper}) gives $z\in\sigma(\Lop_{t,\alpha_0})$. Since $|z|>1>\kappa_t$, part~(\ref{op:spectralpersistence}) of the same proposition gives, for all sufficiently large $j$, spectral values $w_j\in\sigma(\Lop_{t,\beta_j})$ with $w_j\to z$. Then $|w_j|>1$ for all large $j$, contradicting \eqref{C} and \eqref{D}. Thus $d_-=d_+$ and monotonicity proves continuity, including at $1$.
\end{proof}
\subsection{A counterexample}
Take the partition intervals
\begin{equation*}
    I_{1} = (1/4, 1),
    \qquad
    I_{n} = \left( 2^{-(n+1)}, 2^{-n} \right), \quad n \geq 2.
\end{equation*}
Set $T = 0$ at all the partition endpoints (including $0$ and $1$), and define $T$ inside the intervals by
\begin{equation*}
    T(x) =
    \begin{cases}
        1-x+16(\frac{3}{8}-x)_{+}^{2}, & x \in I_{1}, \\
        2-2^{n+1}x, & x \in I_{n}, \ n \geq 2,
    \end{cases}
\end{equation*}
where $u_{+} = \max\{u, 0\}$. It can be checked that $T\bigg|_{I_n}$ is continuous and strictly decreasing and maps $I_n$ onto $(0,1)$, and the only accumulation point of the partition endpoints is $0$. However, the following proposition holds for such a system.
\begin{proposition}
  We have
  \[d\left(\frac{1}{2}\right)\le \frac{2}{3},\qquad d(\alpha)=1,\ \forall\alpha>\frac{1}{2}.\]
\end{proposition}
\begin{proof}
  When $\alpha = 1/2$, if $x \in I_{1}$ and $x < 1/2$, then $T(x) \geq 1 - x > 1/2$. Hence only the fixed point $1/2$ can survive in the first branch. Apart from countably many orbits that eventually land at $1/2$ or at a partition endpoint, surviving orbits can only use the tail branches $n \geq 2$. The inverse maps of these branches are
  \begin{equation*}
    \phi_{n}(y) = 2^{-n} - 2^{-(n+1)} y.
  \end{equation*}
  Let $r_{n} = |\phi_{n}'| = 2^{-(n+1)}$ and let $K$ be the limit set using only these branches. Since
  \begin{equation*}
    \sum_{n=2}^{\infty} r_{n}^{2/3} = \frac{1}{4 \left( 1 - 2^{-2/3} \right)} < 1,
  \end{equation*}
  the cover by cylinder intervals of length $m$ satisfies
  \begin{equation*}
    \sum_{n_{1}, \ldots, n_{m} \geq 2}
    \left| \phi_{n_{1}} \circ \cdots \circ \phi_{n_{m}} ([0,1]) \right|^{2/3}
    = \left( \sum_{n=2}^{\infty} r_{n}^{2/3} \right)^{m}
    \longrightarrow 0.
  \end{equation*}
  Hence,
  \[d_{T}(1/2) = \dim_{\mathrm{H}} K \leq 2/3.\]
  On the other hand, for any $\alpha > 1/2$, choose
  \begin{equation*}
    0 < \varepsilon < \min \left\{ \alpha - 1/2, \, 1/8 \right\}.
  \end{equation*}
  The interval $J = [1/2 - \varepsilon, 1/2 + \varepsilon]$ lies within the reflection branch and satisfies $T(J) = J$ and $J \subset [0,\alpha]$. Therefore $J \subset E(\alpha)$, and thus $d(\alpha) = 1$.
\end{proof}
\section{The proof of Theorems \ref{sejh} and \ref{sews}}
In this section, we will prove Theorems \ref{sejh} and \ref{sews}. First, we prove a proposition which links $S(\alpha)$ and $E(\alpha)$.
\begin{proposition}\label{rcf}
  Let $0<\alpha <1$ and let $T$ be the Gauss map. Then, we have
  \[S(\alpha)=\left\{x=[a_0;a_1,\cdots]\in \mathbb{R}\setminus \mathbb{Q}: T^k(x-a_0)<\alpha,\ \forall k\ge 0\right\}.\]
\end{proposition}
\begin{proof}
  Denote $A=\left\{x=[a_0;a_1,\cdots]\in \mathbb{R}\setminus \mathbb{Q}: T^k(x-a_0)<\alpha,\ \forall k\ge 0\right\}$.
  \par
  First, we prove $A\subset S(\alpha)$. For any $x\in A$, we will prove $x\in S(\alpha)$. Recall that $x=[b_0;\varepsilon_1b_1,\varepsilon_2b_2,\cdots]_{\alpha}$, where $b_0$ is the unique integer such that $x-b_0\in [\alpha-1,\alpha),\ \varepsilon_n=\text{sgn}\left(T_{\alpha}^{n-1}(x-b_0)\right)$ and
  \[b_n=\left\lfloor \left|\frac{1}{T^{n-1}_{\alpha}(x-b_0)}\right|+1-\alpha\right\rfloor.\]
  Now, we prove $a_{n}=b_{n},\ \varepsilon_{n}=1$ and $T^n_{\alpha}(x-b_0)=T^n(x-a_0)$ for all $n\ge 1$ by induction on $n$.
  \par
  First, we prove $a_{2}=b_{2},\ \varepsilon_{1}=1$ and $T_{\alpha}(x-b_0)=T(x-a_0)$. If $b_0\ge a_0+1$, then $x-b_0\le x-a_0-1<\alpha-1$. If $b_0\le a_0-1$, then $x-b_0\ge x-a_0+1>1$. Since $x-b_0\in [\alpha-1,\alpha)$, then $b_0=a_0$, which implies $\varepsilon_1=\text{sgn}(x-b_0)=\text{sgn}(x-a_0)=1$. Hence,
  \[T_{\alpha}(x-b_0)=\frac{1}{x-a_0}-\left\lfloor\frac{1}{x-a_0}\right\rfloor=T(x-a_0).\]
  Since $T(x-a_0)<\alpha$, we have 
  \[\left\{\frac{1}{x-a_0}\right\}+1-\alpha<1.\]
  Note that for any $\beta,\theta\in\mathbb{R}$, we have 
  \begin{align}\label{d}
    \left\lfloor \beta+\theta\right\rfloor=\left\lfloor \beta\right\rfloor\ \text{if and only if}\ \{\beta\}+\theta<1.
  \end{align}
  Then,
  \[b_1=\left\lfloor\left|\frac{1}{x-b_0}\right|+1-\alpha\right\rfloor=\left\lfloor\frac{1}{x-a_0}+1-\alpha\right\rfloor=\left\lfloor\frac{1}{x-a_0}\right\rfloor=a_1.\]
  Since $T^2(x-a_0)<\alpha$, then by same arguments above, we have 
  \[b_2=\left\lfloor\left|\frac{1}{T_{\alpha}(x-b_0)}\right|+1-\alpha\right\rfloor=\left\lfloor\frac{1}{T(x-a_0)}+1-\alpha\right\rfloor=\left\lfloor\frac{1}{T(x-a_0)}\right\rfloor=a_2.\]
  \par
  Suppose $a_{k+1}=b_{k+1},\ \varepsilon_k=1$ and $T^k_{\alpha}(x-b_0)=T^k(x-a_0)$ for all $1\le k\le n$. Now, we prove $a_{n+2}=b_{n+2},\ \varepsilon_{n+1}=1$ and $T^{n+1}_{\alpha}(x-b_0)=T^{n+1}(x-a_0)$. For any $1\le k\le n$, since $\varepsilon_k=1$, then 
  $T_{\alpha}^{k-1}(x-b_0)>0$. Hence, by definition of $T_{\alpha}$,
  \begin{align}\label{c}
    T_{\alpha}^{n+1}(x-b_0)=\frac{1}{T_{\alpha}^{n}(x-b_0)}-b_{n+1}=\frac{1}{T_{\alpha}^{n}(x-b_0)}-a_{n+1}=\frac{1}{T_{\alpha}^{n}(x-a_0)}-\left\lfloor\frac{1}{T^{n}(x-a_0)}\right\rfloor.
  \end{align}
  Therefore, by induction we have $T_{\alpha}^{n+1}(x-b_0)=T^{n+1}(x-a_0)$. Hence, we have $\varepsilon_{n+1}=\text{sgn}\left(T_{\alpha}^n(x-b_0)\right)=1$. Since $T^{n+2}(x-a_0)<\alpha$, by the same arguments in the proof of $b_1=a_1$, we have
  \[b_{n+2}=\left\lfloor\left|\frac{1}{T_{\alpha}^{n+1}(x-b_0)}\right|+1-\alpha\right\rfloor=\left\lfloor\frac{1}{T^{n+1}(x-a_0)}+1-\alpha\right\rfloor=a_{n+2}.\]
  Hence, we have $a_{n}=b_{n}$ and $\varepsilon_n=1$ for all $n\ge 1$, which implies $A\subset S(\alpha)$. 
  \par
  Now, we prove $S(\alpha)\subset A$. For any $x\in S(\alpha)$, we will prove $x\in A$. Since $a_n=b_n,\ \forall n\ge 0$, then $T_{\alpha}^n(x-b_0)=T^n(x-a_0),\ \forall n\ge 0$. For any $n\ge 1$, since $\varepsilon_{n}=1$, then 
  \[b_{n}=\left\lfloor\frac{1}{T_{\alpha}^{n-1}(x-b_0)}+1-\alpha\right\rfloor=\left\lfloor\frac{1}{T^{n-1}(x-a_0)}+1-\alpha\right\rfloor.\]
  Since $a_{n}=b_{n}$, using \eqref{d}, we have 
  \[\left\{\frac{1}{T^{n-1}(x-a_0)}\right\}<\alpha,\]
  which implies $T^{n}(x-a_0)<\alpha$ for all $n\ge 1$. Since $x-a_0=x-b_0$, we also have $x-a_0<\alpha$, which completes the proof.
\end{proof}
\begin{proof}[Proof of Theorem \ref{sejh}]
  We first prove that the set $S(\alpha)$ does not change on $(C_i(n))^{\circ}$ for any $n\ge 1$ and $i=1,2,3$. For any $\alpha\in (C_i(n))^{\circ}$, there exists $\delta>0$, such that $(\alpha-\delta,\alpha+\delta)\subset (C_i(n))^{\circ}$. By Proposition \ref{rcf},
  \[E(\alpha-\delta)\setminus \mathbb{Q}\subset S(\alpha)\cap [0,1]\subset E(\alpha)\setminus \mathbb{Q}.\]
  Since $(\alpha-\delta,\alpha+\delta)\subset (C_i(n))^{\circ}$, by Theorem \ref{jh}, we have $E(\alpha-\delta)=E(\alpha)$.
  Hence, we have $S(\alpha)\cap [0,1]=E(\alpha)\setminus \mathbb{Q}$ for all $\alpha\in (C_i(n))^{\circ}$.
  Since 
  \[S(\alpha)\cap [N,N+1]=\{x+N:x\in S(\alpha)\cap [0,1]\},\]
  we complete the proof.
\end{proof}
\begin{proof}[Proof of Theorem \ref{sews}]
  Since 
  \[S(\alpha)=\bigcup_{n\in\mathbb{Z}}\{x+n: x\in S(\alpha)\cap [0,1]\},\]
  then
  \[\dim_{H}S(\alpha)=\dim_{H}\left(S(\alpha)\cap [0,1]\right).\]
  Note that, by definition
  \[S(\alpha)\cap [0,1]\subset E(\alpha)\subset S(\alpha)\cap [0,1]\cup \mathbb{Q}\cup\left(\bigcup_{k=0}^{+\infty}\left\{x\in [0,1]: T^k(x)=\alpha\right\}\right).\]
  Since $\left\{x\in [0,1]: T^k(x)=\alpha\right\}$ is countable for any $k\ge 0$, then we complete the proof. 
\end{proof}

\par
\textbf{Acknowledgements}: Lingmin Liao was partially supported by the Fundamental Research Funds for the Central Universities of China (1301/413100147), and by Beijing Natural Science Foundation (JQ26002).

%
%\noindent {\bf Data Availability:}
%Data sharing is not applicable to this article as no datasets were generated or analyzed during the current study.
%
%\vspace*{5ex}
%
%\noindent {\bf Conflict of interest:}  On behalf of all authors, the corresponding author states that there is no conflict of interest.
%
%\vspace*{5ex}

 { }

%%%%%%%%%%%%%%%%%%%%%%%%%%%%%%%%%%%%%%%%%%%%%

\vspace*{10ex}

\noindent Cheng Liu: School of Mathematics and Statistics,
Wuhan University, %Universit\'e Paris-Est Créteil CNRS, LAMA, F-94010 Créteil, France
\vspace{-2mm}

\noindent\phantom{Cheng Liu: }Bayi Road 299, Wuchang District, Wuhan, China
% \&
%Universit\'e Gustave Eiffel, LAMA, F-77447 Marne-la-Vallée, France
%\vspace{0mm}

\noindent\phantom{Cheng Liu: }e-mail: 2023202010021@whu.edu.cn %lingmin.liao@u-pec.fr

%%%%%%%%%%%%%%%%%%%%%%%%%%%%%%%%%%%%%%%%%

\vspace{5mm}

\noindent Lingmin Liao: School of Mathematics and Statistics,
Wuhan University, %Universit\'e Paris-Est Créteil CNRS, LAMA, F-94010 Créteil, France
\vspace{-2mm}

\noindent\phantom{Lingmin Liao: }Bayi Road 299, Wuchang District, Wuhan, China
% \&
%Universit\'e Gustave Eiffel, LAMA, F-77447 Marne-la-Vallée, France
%\vspace{0mm}

\noindent\phantom{Lingmin Liao: }e-mail: lmliao@whu.edu.cn %lingmin.liao@u-pec.fr

\end{document}